\documentclass[12pt]{amsart}

\usepackage{amsmath}
\usepackage{chngcntr}
\usepackage{amsthm}
\usepackage{amsfonts}
\usepackage{amssymb}
\usepackage[margin=1.0in]{geometry}
\usepackage{tabularx}
\usepackage{array}
\usepackage{caption}
\usepackage{multicol}
\usepackage{graphicx}
\usepackage{float}

\usepackage{color}
\usepackage{xcolor}
\usepackage{cool}
\usepackage{fancyhdr}
\usepackage{tikz}
\usepackage{pgfplots} 
\usepackage{pgfplotstable} 
\graphicspath{{figures/}}
\usepackage[colorlinks=true]{hyperref}
\usepackage{xcolor}

\newtheorem{thm}{Theorem}[section]
\newtheorem{prop}[thm]{Proposition}
\newtheorem{lem}[thm]{Lemma}
\newtheorem{cor}[thm]{Corollary}

\theoremstyle{definition}
\newtheorem{defn}[thm]{Definition}

\theoremstyle{remark}

\newtheorem{remark}[thm]{Remark}

\numberwithin{equation}{subsection}

\newcommand{\R}{\mathbb{R}}  
\newcommand{\C}{\mathbb{C}} 
\newcommand{\Z}{\mathbb{Z}}
\newcommand{\N}{\mathbb{N}}
\newcommand{\T}{\mathbb{T}}

\newcommand{\TA}{\partial_\theta^\alpha}
\newcommand{\OB}{\partial_\omega^\beta}
\newcommand{\TD}{\partial_t^\delta}
\newcommand{\IB}{\partial_I^\beta}

\newcommand{\Cm}{\mathcal{C}^\infty_{M,\tilde{M}}(X\times Y)}

\colorlet{linkcolour}{black}
\colorlet{urlcolour}{blue}
\hypersetup{colorlinks=false,
	linkcolor=red,
	citecolor=green,
	urlcolor=magenta}

\begin{document}


\title{KAM Hamiltonians are not quantum ergodic}


\author{Se\'{a}n Gomes}
\address{Department of Mathematics, Northwestern University, 
Chicago, IL}
\email{sean.p.gomes@gmail.com}






\begin{abstract}
	We show that under generic conditions, the quantisation of a $1$-parameter family of KAM perturbations $P(x,\xi;t)$ of a completely integrable and Kolmogorov non-degenerate Gevrey smooth Hamiltonian is not quantum ergodic, at least for a full measure subset of the parameter $t\in (0,\delta)$.
\end{abstract}

\maketitle
\tableofcontents
\newpage

\newpage



\section{Introduction}
\subsection{Hamiltonian Dynamics}
Let $M$ be a compact boundaryless Riemannian $G^\rho$ smooth manifold of dimension $n\geq 2$, and let $P(x,\xi)\in \mathcal{C}^\infty (T^*M)$ be a completely integrable Hamiltonian with $P(x,\xi)\rightarrow \infty$ as $|\xi|\rightarrow \infty$. Complete integrability is the assumption that there exist $n$ functionally independent conserved quantities of the Hamiltonian flow that are pairwise in involution.

The Liouville-Arnold theorem asserts that we can locally choose symplectomorphisms
\begin{equation}
\chi:D \times \T^n \rightarrow T^* M
\end{equation}
such that the transformed Hamiltonian
\begin{equation}
H^0=P\circ \chi 
\end{equation}
is independent of $\theta$. It follows that the Hamiltonian flow is quasi-periodic and constrained to $n$-dimensional Lagrangian tori, given in local coordinates by
\begin{equation}
\dot{I}=0;\qquad \dot{\theta}=\nabla H(I).
\end{equation}
Under the Kolmogorov non-degeneracy condition $\det(\nabla_I^2 H)\neq 0$, we can locally index the invariant Lagrangian tori $\Lambda_\omega$ by the frequency $\omega=\nabla_I H$ of their quasi-periodic motion.

If we now consider a smooth one-parameter family of perturbed Hamiltonians given by $H(\theta,I;t)$ in action-angle coordinates with $H(\theta,I;0)=H^0(I)$, a natural question is whether or not any of the tori $\Lambda_\omega$ persist for sufficiently small $t$.
This question was resolved positively by Kolmogorov, Arnold, and Moser \cite{kolmogorov}\cite{arnold2}\cite{moser}. In particular, they established that the Lagrangian invariant tori corresponding to all but a $o(1)$ symplectic measure subset of frequencies survive this perturbation as the size of the perturbation tends to zero. 

In particular peristing tori are those with frequencies $\omega$ in a set $\Omega_\kappa$ determined by the Diophantine condition \eqref{diophantine}, where $\tau>n-1$ is fixed and the choice of $\kappa$ then dictates the measure of the union of preserved tori.

The paper \cite{popovkam} uses a local version of the KAM theorem to construct a Birkhoff normal form for Gevrey class Hamiltonians $H$ about $\Lambda$. This normal form generalises the notion of ``action-angle" variables of a completely integrable Hamiltonian as discussed in \cite{arnoldmechanics}. As a consequence of the normal form construction, Popov obtains an effective stability result for the Hamiltonian flow near the union of remaining invariant tori. The natural setting for the estimates is that of Gevrey regularity. This work generalises earlier work in \cite{popov1} and \cite{popov2} where a Birkhoff normal form is constructed  for real analytic Hamiltonians.

\subsection{Quantum Ergodicity}
We now consider the quantisation of a KAM Hamiltonian system given by a family of self-adjoint semiclassical pseudodifferential operators
\begin{equation}
\label{basicform}
\mathcal{P}_h(t)=\sum_{j=0}^m P_j(x,hD;t)h^j
\end{equation}
acting on the half-density bundle $\mathcal{C}^\infty(M,\Omega^{1/2})$ with principal symbol $P_0$ equal to the KAM Hamiltonian $P(x,\xi;t)$, and  subprincipal symbol $P_1=0$. We assume that $\mathcal{P}_h(t)$ is elliptic and self-adjoint, with fixed positive differential order. The operator $\mathcal{P}_h(t)$ then has an orthonormal basis of eigenfunctions $u_j(t;h)$ and corresponding real eigenvalues $E_j(t;h)\rightarrow \infty$ for each fixed $t,h$.

The Bohr correspondence principle asserts that aspects of the classical dynamics should be reflected in the spectral theory of $\mathcal{P}_h(t)$ in the semiclassical limit $h\rightarrow 0$. A rigorous manifestation of this correspondence principle is the celebrated quantum ergodicity theorem, due to \cite{schnirelman}\cite{verdiere}\cite{zelditch}, which asserts that billiards with ergodic geodesic flow have eigenfunctions satisfying a quantum notion of equidistribution, made precise using the machinery of pseudodifferential operators.

We work with a semiclassical formulation of quantum ergodicity. Let $d\mu_E$ denote the measure on the energy surface $\Sigma_E=p^{-1}(E)$ induced by the symplectic measure $|d\xi \wedge dx|$ on $T^*M$ by
\begin{equation}
|d\mu_E\wedge dE| = |d\xi \wedge dx|.
\end{equation}
If a Hamiltonian  $p(x,\xi)\in\mathcal{C}^\infty(T^*M)$ generates an ergodic Hamiltonian flow on every energy surface $\Sigma_E$ with $E\in [a,b]$ and  $dp|_{p^{-1}([a,b])}\neq 0$, then for any semiclassical pseudodifferential operator $A$ of semiclassical order $0$, we have 
\begin{equation}
\label{eq:qedef}
h^n\sum_{E_j(h)\in [a,b]} \left|\langle A_h u_j(h),u_j(h) \rangle-\frac{1}{\mu_{E_j}(\Sigma_{E_j})}\int_{\Sigma_{E_j}} \sigma(A) \, d\mu_{E_j}\right|^2\rightarrow 0.
\end{equation}
This formulation of the quantum ergodicity theorem is a straightforward consequence of the sharper formulation in \cite{HRM}, or \cite{DYG}, in which the statement is localised to $O(h)$ energy bands.
From \eqref{eq:qedef}, a standard diagonal argument introduced in \cite{colin} shows that
\begin{equation}
\label{eq:qecons}
\langle A_h u_j(h),u_j(h) \rangle\rightarrow \frac{1}{\mu_{E_j}(\Sigma_{E_j})}\int_{\Sigma_{E_j}} \sigma(A) \, d\mu_{E_j}
\end{equation}
uniformly for a family $\Lambda(h)\subset \{E_j(h)\in [a,b]\}$ of full-density, in the sense that
\begin{equation}
\frac{\# \Lambda(h)}{\#\{E_j(h)\in [a,b]\}}\rightarrow 1.
\end{equation} 
We say that a semiclassical pseudodifferential operator of the form \eqref{basicform} is \emph{quantum ergodic} if its eigenfunctions satisfy \eqref{eq:qedef}.

In the appendix to \cite{marklof}, Zelditch raises the question of \textit{converse quantum ergodicity}: To what extent is it possible for non-ergodic Hamiltonian systems such as those in the KAM regime to have quantum ergodic quantisations?
In the extreme situation of quantum complete integrability, rigorous results on eigenfunction microlocalisation onto unions of Lagrangian tori have been established in \cite{toth}, which clearly rules out quantum ergodicity.
In the intermediate regimes between complete integrability and ergodicity, fewer rigorous results on the question of converse quantum ergodicity are known. In the appendix to \cite{marklof}, Zelditch shows that the ``pimpled spheres", which are $S^2$ with a  metric deformed polar cap are not quantum ergodic, exploiting the periodicity of the flow in a strong way. In \cite{gutkin} it is shown that the ``racetrack billiard" is quantum ergodic but not ergodic, with phase space splitting into two disjoint invariant sets of equal measure.

In this paper, we consider families of self-adjoint and uniformly elliptic semiclassical pseudodifferential operators 
\begin{equation}
\label{operatorP}
\mathcal{P}_h(t)=\sum_{j=0}^m P_j(x,hD;t)h^j
\end{equation}
with real-valued full symbol in the Gevrey class $S_\ell(T^*M)$ from Definition \ref{selldef}, smooth in the parameter $t$, where $\ell=(\rho,\mu,\nu)$, with $\rho(\tau+n)+1>\mu>\rho'=\rho(\tau+1)+1$ and $\nu=\rho(\tau+n+1)$. Furthermore, we assume $\mathcal{P}_h(t)$ acts on half-densities in $\mathcal{C}^\infty(M;\Omega^{1/2})$ with principal symbol $P_0(x,\xi;t)$ completely integrable and non-degenerate at $t=0$, and with vanishing subprincipal symbol.
As KAM dynamics are far from ergodic dynamics in character, the Bohr correspondence principle suggests that $\mathcal{P}_h(t)$ is typically \emph{not} quantum ergodic, and that under generic conditions on the perturbation, there could exist sequences of eigenfunctions for $\mathcal{P}_h(t)$ with semiclassical mass entirely supported on individual invariant tori.

This localisation has been proven for quasimodes in the paper \cite{popovquasis}, where semiclassical Fourier integral operators were used to construct a quantum Birkhoff normal form for a class of semiclassical pseudodifferential operators $\mathcal{P}_h(t)$. This quantum Birkhoff normal form is used to obtain a family of quasimodes microlocalised near the union of KAM Lagrangian tori of a Hamiltonian associated to $\mathcal{P}_h$. A similar construction was previously made in \cite{colin}, which establishes the existence of quasimodes microlocalised near the Lagrangian tori of a completely integrable Hamiltonian on a compact smooth manifold. 

As pointed out by Zelditch \cite{zelditchnote} however, the passage from quasimode microlocalisation statements to microlocalisation statements for genuine eigenfunctions typically requires information on the spectral concentration of the operator in question.

One way in which this information can be obtained is by considering the spectral flow of $\mathcal{P}_h(t)$ in an analytic parameter $t$ as in this paper. The Hadamard variational formula allows us to rule out spectral concentration for full measure $t$, given suitable information on the expectation of the quantum observable
\begin{equation}
\langle \mathcal{P}_h'(t)u_j(t;h),u_j(t;h)\rangle
\end{equation}
which can be obtained from conditions like \eqref{eq:qecons}.
One can then draw conclusions about eigenfunction microlocalisation from those about quasimode microlocalisation.

In \cite{hassellque}, this technique was exploited to obtain the existence of a sequence of Laplacian eigenfunctions on the Bunimovich stadium that does not equidistribute, at least for a full measure set of aspect ratios. This strategy was also exploited in \cite{mushroom}, where the author establishes a weak form of Percival's conjecture for the mushroom billiard.

It is the purpose of this paper to use the same technique to show that quantisations of KAM Hamiltonian systems in the sense of \eqref{operatorP} are typically \emph{not} quantum ergodic, at least for full measure $t\in (0,\delta)$. 

We follow Popov \cite{popovquasis} in working in the category of Gevrey regularity for our Hamiltonian $P$, due to the availability of explicit and full  details of the quantum Birkhoff normal form construction in this setting.

\subsection{Statement of results}
The following is the main result of this paper.
\begin{thm}
	\label{thm:main}
	\label{nonqe}
	Suppose $M$ is a compact boundaryless $G^\rho$ manifold and $\mathcal{P}_h(t)$ is a family of self-adjoint elliptic semiclassical pseudodifferential operators acting on $\mathcal{C}^\infty (M;\Omega^{1/2})$ with fixed positive differential order such that 
	
	\begin{enumerate}
		\item The operator $\mathcal{P}_h(t)$ has full symbol real-valued, smooth in $t$, and in the Gevrey class $S_\ell(T^*M)$ from Definition \ref{selldef} where $\ell=(\rho,\mu,\nu)$, with $\rho(\tau+n)+1>\mu>\rho'=\rho(\tau+1)+1$ and $\nu=\rho(\tau+n+1)$;
		\item The principal symbol $P_0(x,\xi;t)$ lies in $G^{\rho,1}(T^*M\times (-1,1))$;
		\item $P_0(x,\xi;0)$ is a completely integrable and non-degenerate Hamiltonian;
		\item The subprincipal symbol of $\mathcal{P}_h(t)$ vanishes;
		\item In an action-angle variable coordinate patch $\T^n\times D$ for the unperturbed Hamiltonian $P_0(x,\xi;0)$, the KAM Hamiltonian can be written as $H(\theta,I;t)=P_0(\cdot,\cdot;t)\circ \chi$, and we define $H^0(I):=H(\theta,I;0)$;
		\item \label{genericcond} The KAM perturbation is such that
		\[\int_{\T^n} \partial_t H(\theta,I;0)\, d\theta\] 
		is nonconstant on some regular energy surface  $\{I\in D:H^0(I)=E\}$ in the action-angle coordinate patch.
	\end{enumerate}
	Then for any regular energy band $P_0^{-1}(E)$ with $E\in[a,b]$, there exists $\delta >0$ such that the family of operators $\mathcal{P}_h(t)$ is not quantum ergodic in $[a,b]$ for full measure $t\in (0,\delta)$.
\end{thm}

\begin{remark}
	Though we choose to work with Gevrey class Hamiltonians, it should be noted that we only require quasimodes for $\mathcal{P}_h(t)$ of order $O(h^{\frac{3n+2}{2}})$ to carry out the arguments in Section \ref{sec:main}. In particular this implies that Theorem \ref{thm:main} should hold in the $\mathcal{C}^\infty$ setting, where $O(h^\infty)$ quasimodes are constructed in \cite{colin}.
\end{remark}

\begin{remark}
	The condition \eqref{genericcond} is a rather mild one. Indeed for Hamiltonian perturbations of the form $H^0(I)+tH^1(\theta,I)$, it is equivalent to the functional independence of $H^0(I)$ and  $\int_{\T^n}H^1(\theta,I)\, d\theta$. This holds for generic choice of $H^1$. 
\end{remark}

\subsection{Examples}

The broad class of operators satisfying the assumptions of Theorem \ref{thm:main} are perturbations of completely integrable Schr\"{o}dinger type operators
\begin{equation}
\mathcal{P}_h=-h^2\Delta_g +V(x).
\end{equation}
In particular, Theorem \ref{thm:main} applies to the case of the semiclassical Laplace-Beltrami operator ($V=0$) on a manifold with perturbed metric $(M,g_t)$, where $(M,g_0)$ has completely integrable and non-degenerate geodesic flow. 

The model example of a completely integrable geodesic flow is that of the flat torus
\begin{equation}
\T^n=\R^n/\Z^n.
\end{equation}
The Hamiltonian that generates the geodesic flow on $\T^n$ can be written as $|I|^2$, where $I\in \R^n$ is dual to the spatial variable $\theta\in \T^n$. This is clearly a non-degenerate and completely integrable Hamiltonian system. Similarly, in \cite{ellipsoid}, it is shown that the geodesic flow on an $n$-axial ellipsoid $E$ is completely integrable and non-degenerate.

Thus the Laplace-Beltrami operator for  metric perturbations of both of these manifolds is covered by by Theorem \ref{thm:main}, provided the generic condition \eqref{genericcond} is satisfied.

\subsection{Outline of paper}

In Section \ref{sec:notation}, we introduce some definitions and notations that are prevalent throughout the paper.

In Section \ref{sec:main}, we prove Theorem \ref{thm:main} by contradiction. We now outline the strategy of the proof. In Section \ref{flowspeedsec}, under the assumptions of \eqref{genericcond} in Theorem \ref{thm:main}, Proposition \ref{quasispeedprop} makes use of the calculation in Section \ref{leading} to obtain an upper bound for the flow speed of a positive density family of the quasi-eigenvalues constructed in Section \ref{gevquasisec}. On the other hand, the assumption of quantum ergodicity of $\mathcal{P}_h(t)$ for large measure $t$ yields an estimate for the variation of exact eigenvalues in \eqref{hvfcons}. The results in this section establish a gap \eqref{discrepancy} between the the flow speed of these quasi-eigenvalues and exact eigenvalues that ensures that individual eigenvalues cannot spend large measure $t\in(0,\delta)$ within $O(h^{n+1})$ distance of any of the quasi-eigenvalues. This is formalised in Section \ref{specflowsec}, where it is deduced that there exists $t_*\in (0,\delta)$ at which there are very few actual eigenvalues within $O(h^{n+1})$ distance of the union of quasi-eigenvalue windows. An elementary spectral theory contradiction is arrived at from this spectral non-concentration, completing the proof.

In Section \ref{sec:bnf}, we construct a Gevrey class Birkhoff normal form for the family of Hamiltonians $P(x,\xi;t)$. The construction is that of Popov \cite{popovkam}, with our only additional concern being establishing the regularity of this Birkhoff normal form construction in the parameter $t$. In Section \ref{leading}, we compute the derivative of the integrable term  $K(I;t)$ of the Birkhoff normal form in the parameter $t$. This is done by applying two KAM iterations to $P(x,\xi;t)$ prior to the application of the Birkhoff normal form construction of Theorem \ref{main1}. 

In Section \ref{sec:qbnf}, we recall the quantum Birkhoff normal form construction of Popov \cite{popovquasis}, formulated in Theorem \ref{main2}. This construction yields a Gevrey family of quasimodes microlocalising on the KAM Lagrangian tori of the Hamiltonian $P(x,\xi;t)$. For the spectral flow arguments in Section \ref{specflowsec} we require that the associated quasi-eigenvalues are smooth in $t$, which is a statement entirely about the symbols of this quantum Birkhoff normal form. 

In Appendix \ref{gevreyappendix}, we introduce the anisotropic classes of Gevrey functions that are used throughout this paper as well as some of their basic properties.

In Appendix \ref{gevsymbsec}, we introduce the semiclassical pseudodifferential calculus for Gevrey class symbols.

In Appendix \ref{analytic}, we collect two elementary assertions about analytic functions.

In Appendix \ref{whitneysec}, we state and prove a version of the Whitney extension theorem for the anisotropic class of Gevrey functions.

\section{Proof of Theorem \ref{thm:main}}
\label{sec:main}
\subsection{Introduction}
\label{6intro}
We begin by assuming that $\mathcal{P}_h(t)$ is a family of operators satisfying the assumptions of Theorem \ref{thm:main}. 

The condition \eqref{genericcond} in Theorem \ref{thm:main} implies that there exists a nonresonant frequency $\omega_0\in \tilde{\Omega}_\kappa$ with associated Lagrangian torus $\Lambda_{\omega_0}$ such that the average of $\partial_t P_0(x,\xi;0)$ over the torus $\Lambda_{\omega_0}$ differs from the average of $\partial_t P_0(x,\xi;0)$ over the associated energy surface \begin{equation}
{\{(x,\xi)\in T^*M :P_0(x,\xi;0)=H^0(I(\omega_0))\}}.
\end{equation}
Moreover, we can ensure that $\Lambda_{\omega_0}$ lies in an arbitrarily small energy window $[a,b]$ about the regular energy $E$ from the condition \eqref{genericcond}. Without loss of generality, the hypotheses of Theorem \ref{thm:main} thus guarantee the existence of what we shall call a \emph{slow torus}. 

\begin{defn}
	\label{slowtorusdef}
	A \emph{slow torus} in the energy band $[a,b]$ for the unperturbed Hamiltonian 
	\begin{equation}
	H(\theta,I;0)=H^0(I)
	\end{equation}
	written in action-angle coordinates, is a Lagrangian invariant torus $\Lambda_{\omega_0}$ with nonresonant frequency $\omega_0\in \tilde{\Omega}_\kappa$ and energy ${H^0(I(\omega_0))\in [a,b]}$ in the notation of Theorem \ref{main1} that satisfies
\begin{equation}
\label{slowtorus}
(2\pi)^{-n}\int_{\T^n} \partial_tH(\theta,I(\omega_0);0)\, d\theta < \inf_{E\in [a,b]}\frac{1}{\mu_E(\Sigma_E)}\int_{\Sigma_E} \partial_t P_0(x,\xi;0)\, d\mu_E
\end{equation}
	at $t=0$.
\end{defn}

We call such a torus a \emph{slow torus} to draw intuition from the special case where $\mathcal{P}_h'(t)$ is a positive operator. In this case, as $t$ evolves, the quasi-eigenvalues associated to such a torus increase as $t$ evolves at a slower rate than the typical increase of eigenvalues at the same energy. The intuition behind this stems from the Hadamard variational formula \eqref{qequantity}, and the fact that the associated quasimodes microlocalise onto $\Lambda_{\omega_0}$. This intuition is confirmed in Section \ref{leading}, by a more careful  analysis of the leading order behaviour as $t\rightarrow 0$ of the integrable term in the Birkhoff normal form established in Theorem \ref{main1}. This discrepancy \eqref{discrepancy} in the spectral flow of genuine eigenvalues and quasi-eigenvalues attached to slow tori leads to the spectral non-concentration statement Proposition \ref{specnonconckamprop}.

We begin by using the slow torus condition and choosing $c>0$ sufficiently small so that
\begin{equation}
\label{firstslowpostori}
(2\pi)^{-n}\int_{\T^n} \partial_tH(\theta,I(\omega_0;0);0)\, d\theta < \inf_{E\in [a,b]}\frac{1}{\mu_E(\Sigma_E)}\int_{\Sigma_E} \partial_tP_0(x,\xi;0)\, d\mu_E-3c
\end{equation}
is satisfied.

As the quantum ergodicity condition \eqref{eq:qedef} is preserved upon passing to energy subintervals, we can assume that $[a,b]$ is an arbitrarily small energy window containing $H^0(I(\omega_0;0))$. In particular, we can scale our interval $[a,b]$ by a small factor $\lambda$ to ensure that the condition
\begin{equation}
\label{qplusqminus}
\sup_{E\in[a,b]}\frac{1}{\mu_E(\Sigma_E)}\int_{\Sigma_E} \partial_tP_0\, d\mu_E - \inf_{E\in [a,b]}\frac{1}{\mu_E(\Sigma_E)}\int_{\Sigma_E} \partial_tP_0\, d\mu_E =: Q_+(0)-Q_-(0) < \epsilon < c.
\end{equation}
is satisfied for any particular $\epsilon<c$. From the regularity of $P_0$, one can achieve this by taking
\begin{equation}
\label{lambdadef}
\lambda=O(\epsilon).
\end{equation}
Through the course of this Section, we will track the size of various small quantities in terms of this $\epsilon$, which we will eventually take small in the proof of Proposition \ref{specnonconckamprop}.

Theorem \ref{bnf2} applies to $H$, and we obtain a family of symplectomorphisms \begin{equation}
\chi\in G^{\rho,\rho',\rho'}(\T^n\times D\times (-1/2,1/2), \T^n\times D)
\end{equation} 
and a family of diffeomorphisms 
\begin{equation}
\label{omegadiffeomainsec}
\omega\in G^{\rho',\rho'}(D\times (-1/2,1/2),\Omega)
\end{equation}
such that 
\begin{equation}
\label{BNF6}
H(\chi(\theta,I;t);t)=K(I;t)+R(\theta,I;t) 
\end{equation}
where $R$ is flat in $I$ at the nonresonant actions $I\in E_\kappa(t)$.
Using the diffeomorphism \eqref{omegadiffeomainsec}, we can define an action map $I\in G^{\rho',\rho'}(\Omega\times (-1/2,1/2))$ implicitly by 
\begin{equation}
\tilde{\omega}=\omega(I(\tilde{\omega};t) ; t)
\end{equation}
and we can use this map to specify the action coordinates of a nonresonant torus with fixed frequency at any $t\in (-1/2,1/2)$ in the Birkhoff normal form furnished by $\chi(\cdot,\cdot;t)$.

We first obtain a positive measure family of slow tori near $\Lambda_{\omega_0}$.
\begin{prop}
	\label{fattori}
	There exists $r>0$ and $\delta>0$ such that for any $\omega\in \overline{\Omega}:= B(\omega_0,r)\cap \tilde{\Omega}_\kappa$, the torus $\Lambda_\omega=\chi(\T^n\times \{I(\omega,t)\})$ has energy  
	\begin{equation}
	\label{inenergywindow}
	K(I(\omega;t),t)\in[a,b]
	\end{equation}
	for all $t\in (0,\delta)$.
	
	In particular, the family of tori
	\begin{equation}
	\label{torifam}
	\Lambda(t):=\bigcup_{\omega\in \overline{\Omega}} \Lambda_{\omega}
	\end{equation}
	is a positive measure family of KAM tori entirely contained within the energy band $[a,b]$.
	
	Moreover, $r$ and $\delta$ can be chosen small enough to ensure
	\begin{eqnarray}
	\nonumber (2\pi)^{-n}\int_{\T^n} \partial_tH(\theta,I(\omega;t);t)\, d\theta &<&  (2\pi)^{-n}\int_{\T^n} \partial_tH(\theta,I(\omega;0);t)\, d\theta+\epsilon\\ &<& \label{slowpostori}\inf_{t\in (0,\delta)}Q_-(t)-2c.
	\end{eqnarray}

	for each $\omega\in \overline{\Omega}$ and each $t\in (0,\delta)$.
	
	We can also choose $\delta>0$ small enough to ensure that
	\begin{equation}
	\label{ergodicboundsclose}
	Q_+-Q_-:=\sup_{t\in (0,\delta)} Q_+(t)- \inf_{t\in (0,\delta)} Q_-(t)<2\epsilon.
	\end{equation}
	
	In particular $r,\delta$ can be taken to be $O(\epsilon)$, with constant independent of $t$ and $h$.
\end{prop}
\begin{proof}
	From the regularity of $\chi,I,$ and $K$ established in Theorem \ref{main1}, it follows that we can take $r=O(\lambda)$ to ensure that \eqref{inenergywindow} is satisfied at $t=0$, where $\lambda=O(\epsilon)$ is as in \eqref{lambdadef}. Similarly, we can ensure that 
	\begin{equation}
	(2\pi)^{-n}\int_{\T^n} \partial_t H(\theta,I;0)\, d\theta <  (2\pi)^{-n}\int_{\T^n} \partial_tH(\theta,I(\omega_0);t)\, d\theta+\epsilon/2
	\end{equation}
	holds for $|I-I(\omega_0)|=O(\lambda)$. Since \eqref{firstslowpostori} is satisfied at $t=0$, it follows that 
	\begin{equation}
	(2\pi)^{-n}\int_{\T^n} \partial_t H(\theta,I(\omega;0);0)\, d\theta <  Q_-(0)-3c+\epsilon/2
	\end{equation}
	for all $\omega\in \overline{\Omega}=B(\omega_0,r)\cap \tilde{\Omega}_\kappa$ upon taking $r=O(\lambda)$.
	
	The regularity of $\chi,I$ and $K$ in the parameter $t$ then allow us to then deduce that \eqref{inenergywindow} and \eqref{slowpostori} are satisfied for $t\in(0,\delta)$, for sufficiently small $\delta>0$ and for each $\omega \in \overline{\Omega}$. In particular, we can take $\delta=O(\lambda)=O(\epsilon)$. 
	
	Finally, the estimate \eqref{ergodicboundsclose} for small $\delta$ follows from the regularity of 
	\begin{equation}
	\frac{1}{\mu_E(\Sigma_E)}\int_{\Sigma_E} \partial_tP_0 d\mu_E
	\end{equation}
	in $t$ and $E$.
\end{proof}

We can now apply the quantum Birkhoff normal form construction outlined in Section \ref{sec:qbnf}. This yields a family of quasimodes  that microlocalise onto the family of KAM tori $\Lambda(t)$ introduced in \eqref{torifam}. In particular, following \ref{gevquasisec}, we take $S(t)=\{I(\omega;t):\omega\in \overline{\Omega}\}$ and define the index set  $\mathcal{M}_h(t)$ as in \eqref{indexsetref}.

We next introduce notation for the union of $h^{n+1}$-width energy windows about the quasi-eigenvalue associated to tori in $\Lambda(t)$.
\begin{equation}
\label{wthdef}
W(t;h):= \bigcup_{m\in \mathcal{M}_h(t)} [K^0(h(m+\vartheta/4),t;h)-h^{n+1},K^0(h(m+\vartheta/4),t;h)+h^{n+1}]
\end{equation}
where $K^0$ is as in Theorem \ref{main2}.

For the sake of brevity, we introduce the notation 
\begin{equation}
\mu_m(t;h):=K^0(h(m+\vartheta/4),t;h)
\end{equation} 
for the quasi-eigenvalues under consideration.

We also introduce the index set
\begin{equation}
\label{Gdef}
G(h)=\{j\in \N:E_j(t)\in [a,b]\textrm{ for some }t\in (0,\delta)\}
\end{equation}
of the eigenvalues that can possibly play a role in the spectral flow considerations in Section \ref{specflowsec}.

To conclude this section, we collect asymptotic estimates for the number of eigenvalues and the number of quasi-eigenvalues that are in the energy window $[a,b]$ as $h\rightarrow 0$.

\begin{prop}
	\label{counting}
	We have the asymptotic estimate
	\begin{equation}
	\label{indexsetcount}
	\#\mathcal{M}_h(t)\sim (2\pi h)^{-n}\mu(\T^n\times \{I(\omega,t):\omega\in \overline{\Omega}\}).
	\end{equation}
	for each $t\in (0,\delta)$.
	
	Furthermore, we have
	\begin{equation}
	\label{eigencount}
	\limsup_{h\rightarrow 0}(2\pi h)^n\#G(h)\leq \mu(\{(x,\xi):P_0(x,\xi;0)\in [a-M\delta,b+M\delta]\})
	\end{equation}
	where $M$ is the uniform bound on spectral flow speed in \eqref{globalboundgevrey} and $G(h)$ is as in \eqref{Gdef}.
	
	Here $\mu$ denotes the symplectic measure $d\xi\, dx$ on $T^*M$.
\end{prop}

\begin{proof}
	The estimate \eqref{indexsetcount} is a consequence from \eqref{indexsetasymp}, and \eqref{eigencount} follows from \eqref{globalboundgevrey} and an application of the semiclassical Weyl law \cite[Theorem 14.11]{zworski}.
\end{proof}

From Proposition \ref{counting}, it follows that we can bound
\begin{equation}
\frac{\#G(h)}{\displaystyle\inf_{t\in (0,\delta)}\# \mathcal{M}_h(t)}
\end{equation}
for $t\in (0,\delta(\epsilon))$ and $h<h_0(\epsilon)$. Moreover, this upper bound is uniform in $\epsilon$. By the nature of their construction in Proposition \ref{fattori}, the quasi-eigenvalues $\mu_m(t;h)$ lie in $[a,b]$ for all $t\in (0,\delta)$.

It is convenient to introduce the subset $\tilde{G}(h)\subset G(h)$ given by
\begin{equation}
\label{tildeGdef}
\tilde{G}(h)=\{j\in \N: E_j(t)\in [a,b]\textrm{ for all }t\in (0,\delta)\}.
\end{equation}
By choosing $\delta(\epsilon) > 0$ appropriately small, we can ensure that a large proportion of eigenvalues that lie in $[a,b]$ for some $t\in (0,\delta)$ lie in $[a,b]$ for all $t\in(0,\delta)$.

\begin{prop}
	\label{tildeGisfat}
	We can choose $\delta(\epsilon)=O(\epsilon^2)$ such that
	\begin{equation}
	\frac{\#\tilde{G}(h)}{\#G(h)}\geq 1-C\epsilon
	\end{equation}
	for all $\epsilon<\epsilon_0$ and $h<h_0(\epsilon)$, where $C>0$ is a constant.
\end{prop}
\begin{proof}
	We can bound
	\begin{equation}
	\label{beforeapplyingsclweyl}
	\frac{\#G(h)}{\#\tilde{G}(h)}\leq \frac{N_h([a+M\delta,b-M\delta])}{N_h([a-M\delta,b+M\delta])}
	\end{equation}
	where $N_h(I)$ counts the semiclassical eigenvalues of the operator $\mathcal{P}_h(0)$ in $I$.	
	Recalling that the interval $[a,b]$ is of scale $\lambda=O(\epsilon)$, it follows that for any choice of $\delta=O(\epsilon^2)$, the ratio of phase space volumes
	\begin{equation}
	\frac{\mu(P_0(x,\xi;0)\in [a-M\delta,a+M\delta]\cup [b-M\delta,b+M\delta])}{\mu(P_0(x,\xi;0)\in[a-M\delta,a+M\delta])}
	\end{equation}
	can be bounded by a constant multiple of $\epsilon$ for all sufficiently small $\epsilon$. Application of the semiclassical Weyl asymptotics to \eqref{beforeapplyingsclweyl} completes the proof.
\end{proof}

\subsection{Eigenvalue and quasi-eigenvalue variation}
\label{flowspeedsec}

We now turn our attention to the variation of quasi-eigenvalues and eigenvalues as $t\in(0,\delta)$ varies. The quasi-eigenvalues can be handled rather explicitly.

\begin{prop}
	\label{quasispeedprop}
	For any all sufficiently small $\delta(\epsilon)>0$ and all $t\in(0,\delta)$, we have 
	\begin{equation}
	\limsup_{h\rightarrow 0}\partial_t \mu_m(t;h)\leq Q_- -c.
	\end{equation}
	for all $m\in \cup_{t\in (0,\delta)}\mathcal{M}_h(t)$ uniformly in $t$.
\end{prop}
\begin{proof}
	From Proposition \ref{bnf2}, we have 
	\begin{equation}
	K_0(I;t)=H^0(I)+t\cdot (2\pi)^{-n}\int_{\T^n} \partial_t H(\theta,I;0)\, d\theta +O(t^{9/8})
	\end{equation}
	for any $I\in D$.
	Hence we have
	\begin{equation}
	\partial_t(K_0(h(m+\vartheta/4);t))< (2\pi)^{-n}\int_{\T^n} \partial_t H(\theta,h(m+\vartheta/4);0) \, d\theta + \epsilon
	\end{equation}
	for all $t\in(0,\delta(\epsilon))$, taking $\delta$ sufficiently small.
	From the definition of $\mathcal{M}_h(t)$, we know that $|h(m+\vartheta/4)-I(\omega;t)|< Lh$ for some $\omega\in \overline{\Omega}$, and so from the regularity of $I$ in $t$ it follows that
	\begin{equation}
	\partial_t(K_0(h(m+\vartheta/4);t)) < (2\pi)^{-n}\int_{\T^n} \partial_t H(\theta,I(\omega;t);t) \, d\theta+\epsilon+O(h)
	\end{equation}
	for some $\omega\in \overline{\Omega}$. This allows us to use \eqref{slowpostori}.
	
	Indeed, we have 
	\begin{eqnarray}
	\partial_t\mu_m(t;h)&=&\partial_t(K^0(h(m+\vartheta/4);t,h))\\
	&=& \partial_t(K_0(h(m+\vartheta/4);t))+O(h)\\
	\Rightarrow \limsup_{h\rightarrow 0} \partial_t\mu_m(t;h)& < & Q_- -2c+\epsilon.
	\end{eqnarray}
\end{proof}
In particular, we can choose $B>0$ and $h_0>0$ such that 
\begin{equation}
\label{Bdefn}
\partial_t \mu_m(t;h)<B< Q_-- c
\end{equation}
for all $t\in (0,\delta)$ and all $h<h_0$.

\begin{remark}
	We have abused notation slightly here by writing $\mu_m(t;h)$ even when $m\notin \mathcal{M}_h(t)$. That is, we track the behaviour of $K^0(h(m+\vartheta/4),t;h)$ even for $t\in(0,\delta)$ such that this does not correspond to a quasi-eigenvalue in our family. This is a necessity due to the rough nature of the set $\{I(\omega;t):\omega\in\overline{\Omega}\}$ of nonresonant actions. Indices $m\in \Z^n$ will typically be elements of $\mathcal{M}_h(t)$ for only $O(h)$-sized $t$-intervals at a time.
\end{remark}

\begin{remark}
	This is the last part of the argument that involves placing an additional restriction on the size of $\delta>0$. 
\end{remark}
We now consider the variation of eigenvalues. For each fixed $h>0$, the operators $\mathcal{P}_h(t)$ comprise an holomorphic family of type A in the sense of \cite{kato} and so we can choose eigenvalues and corresponding eigenprojections holomorphic in the parameter $t$. Thus if at each time $t$ we order our eigenpairs $E_j(t;h)$ in order of increasing energy, by holomorphy it follows that $E_j$ will be continuous, piecewise smooth, and have multiplicity $1$ for all but finitely many $t\in (0,\delta)$. On this cofinite set, we have
\begin{eqnarray}
\dot{E}_j(t;h)&=&\label{qequantity}\langle\mathcal{P}'_h(t)u_j(t;h),u_j(t;h)\rangle.
\end{eqnarray}
from $(u_j)$ being an orthonormal basis. We will control \eqref{qequantity} using our assumption of quantum ergodicity. 

To this end, we now suppose for the sake of contradiction that there exists a positive measure set $\mathcal{B}\subset (0,\delta)$ such that $\mathcal{P}_h(t)$ is quantum ergodic in the sense of \eqref{eq:qedef} for every $t\in \mathcal{B}$.
\begin{prop}
	\label{pointwisebound}
	For every $t\in \mathcal{B}$ and $\epsilon>0$, there exists $h_0(t,\epsilon)$ such that for all $h<h_0$, we have
	\begin{equation}
	|\langle \mathcal{P}_h'(t)u_j,u_j \rangle-\int_{\Sigma_{E_j}}\partial_tP_0\, d\mu_{E_j} |<\epsilon
	\end{equation}
	for a family of indices $S(t;h)\subset \{j\in\N:E_j(t;h)\in [a,b]\}$ with
	\begin{equation}
	\frac{\#S(t;h)}{\{j\in\N:E_j(t;h)\in[a,b]\}}>1-\epsilon.
	\end{equation}
\end{prop}
\begin{proof}
	This is a direct application of \eqref{eq:qecons}.
\end{proof}

We also note that we have a global in time bound 
\begin{equation}
\label{globalboundgevrey}
E_j'(t)\leq M<\infty
\end{equation}
from differentiation of the expression
\begin{equation}
E_j(t)=\langle \mathcal{P}_h(t)u_j(t),u_j(t) \rangle
\end{equation}
and using a routine elliptic parametrix construction that is uniform in $t\in(0,1)$ to bound the quantity
\begin{equation}
\langle \mathcal{P}_h'(t)u_j(t),u_j(t)\rangle
\end{equation}
given that $E_j(t)$ lies in a fixed energy band $[a,b]$.

Recalling \eqref{qplusqminus}, Proposition \ref{pointwisebound} implies that 
\begin{equation}
\langle\mathcal{P}_h'(t)u_j,u_j \rangle \in [Q_--\epsilon,Q_++\epsilon]
\end{equation}
for all $j\in S(t,h)$ such that $E_j$ is smooth at $t$, and all $h<h_0(t,\epsilon)$.

Now, from the outer regularity of the Lebesgue measure, we may then choose a subinterval $J\subset (0,\delta)$ such that
\begin{equation}
\frac{m(\mathcal{B}\cap J)}{m(J)}>1-\epsilon.
\end{equation}
We can then apply the monotone convergence theorem to upgrade Proposition \ref{pointwisebound} for $t\in \mathcal{B}$ to a statement that is uniform in a large measure subset of $J$.
\begin{prop}
	\label{almostuniformgevrey}
	There exists a subset $\tilde{\mathcal{B}}\subseteq \mathcal{B}\cap J$ and a $h_0>0$ such that
	\begin{equation}
	\frac{m(\tilde{\mathcal{B}})}{m(J)}>1-2\epsilon
	\end{equation}
	and for any $h<h_0(\epsilon)$ and any $t\in \tilde{\mathcal{B}}$, there exists a subset 
	\begin{equation}
	Z(t,h)\subset \{j\in \N: E_j(t,h)\in [a,b]\}
	\end{equation} 
	such that
	\begin{equation}
	\frac{\#Z(t,h)}{\#\{j\in \N:E_j(t,h)\in [a,b]\}} >  1- 2\epsilon\quad \forall 0<h < h_0
	\end{equation}
	and 
	\begin{equation}
	\label{enlargedinterval}
	\langle\mathcal{P}_h'(t)u_j,u_j \rangle \in [Q_--\epsilon,Q_++\epsilon]\quad \forall j \in Z(t,h).
	\end{equation}
\end{prop}
In light of Proposition \ref{almostuniformgevrey} , we re-define $Q_-,Q_+$ to be the endpoints of the enlarged interval in \eqref{enlargedinterval}. Hence
\begin{equation}
\label{hvfcons}
\langle\mathcal{P}_h'(t)u_j,u_j \rangle \in [Q_--\epsilon,Q_++\epsilon]\quad \forall j \in Z(t,h).
\end{equation}
In terms of the re-defined $Q_-,Q_+$, we have
\begin{equation}
\label{discrepancy}
Q_--B>c-\epsilon>0.
\end{equation}
and so we have established a discrepancy between the typical speed of eigenvalue flow and the upper bound for the speed of quasi-eigenvalue flow.

\subsection{Spectral non-concentration}
\label{specflowsec}

We can now complete the proof of Theorem \ref{nonqe} by proving a spectral non-concentration result that follows from the results of Section \ref{flowspeedsec}.

\begin{prop}
	\label{specnonconckamprop}
	For sufficiently small $\epsilon>0$, there exists $t_*\in J$ such that 
	\begin{equation}
	\label{specnonconckam}
	\frac{N(t_*;h)}{\#\mathcal{M}_h(t_*)}<1/2
	\end{equation}
	for a sequence $h_j\rightarrow 0$, where 
	\begin{equation}
	\label{Ndef}
	N(t;h):=\#\{j\in\N:E_j(t;h)\in W(t;h)\}.
	\end{equation}
	is the number of exact eigenvalues lying in the union $W(t,h)$ of the  quasi-eigenvalue windows as introduced in \eqref{wthdef}.
\end{prop}
\begin{proof}
    The method of proof is by averaging in $t$ and using Proposition \ref{almostuniformgevrey} to show that most individual eigenfunctions cannot lie in $W(t,h)$ for a significant proportion of $t\in J$.
	
	We begin by defining
	\begin{eqnarray}
	A_j(h)&=&\{t\in J : E_j(t;h)\in [a,b]\}\\
	B_j(h)&=&\{t\in J : j\in Z(t;h)\}\\
	C_j(h)&=&\{t\in J:E_j(t;h)\in W(t;h)\}.
	\end{eqnarray}

	From Proposition \ref{almostuniformgevrey}, for each $t\in \tilde{\mathcal{B}}$ we have
	\begin{equation}
	\sum_{j\in \N} 1_{B_j} \geq (1-2\epsilon)	\sum_{j\in \N}1_{A_j}
	\end{equation}
	for $h<h_0(\epsilon)$.
	Integrating, we obtain
	\begin{equation}
	\sum_{j\in\N}\int_{\tilde{B}} 1_{B_j}\, dt \geq (1-2\epsilon)\sum_{j\in \N}\int_{\tilde{B}} 1_{A_j}\, dt.
	\end{equation}
	Hence
	\begin{eqnarray}
	\sum_{j\in\N}\int_J 1_{B_j}\, dt &\geq& (1-2\epsilon)\sum_{j\in \N}\left(\int_J 1_{A_j}\, dt-\int_{J\setminus\tilde{B}}1_{A_j}\, dt\right)\\
	&\geq &\label{meow}(1-2\epsilon)\sum_{j\in \N}\left(\int_J 1_{A_j}\, dt-2\epsilon m(J)\right)
	\end{eqnarray}
	which can be rewritten as
	\begin{equation}
	\sum_{j\in\N}m(B_j) \geq (1-2\epsilon)\sum_{j\in \N}(m(A_j)-2\epsilon m(J)).
	\end{equation}
	From the definitions \eqref{Gdef} and \eqref{tildeGdef}, we know that $m(A_j)>0$ only if $j\in G(h)$ and $m(A_j)=m(J)$ if $j\in\tilde{G}(h)$.
	Thus we can estimate
	\begin{eqnarray}
	\frac{1}{\#G(h)}\sum_{j\in \N} m(B_j) &\geq& (1-2\epsilon)(\frac{\#\tilde{G}(h)}{\#G(h)}-2\epsilon)m(J)\\
	&\geq& 	(1-\epsilon)(1-O(\epsilon))m(J)\\
	&=: &(1-\eta)m(J)
	\end{eqnarray}
	where $\limsup_{h\rightarrow 0}\eta(\epsilon;h)=o_\epsilon(1)$.

	Consequently we have 
	\begin{equation}
	m(B_j)\geq (1-\eta^{1/2})m(A_j)
	\end{equation}
	for a subfamily $\mathcal{F}(h)\subset \tilde{G}(h)$ with
	\begin{equation}
	\frac{\#\mathcal{F}(h)}{\#G(h)}\geq 1-\eta^{1/2}-O(\epsilon)
	\end{equation}
	in the limit $h\rightarrow 0$, where we have made use of Proposition \ref{tildeGisfat}.
	
	Taking $E(t;h):=E_j(t;h)$ for some $j\in \mathcal{F}$, the bound from the Hadamard variational formula \eqref{hvfcons} yields
	\begin{equation}
	\label{eigenjumplower}
	E(t_2;h)-E(t_1;h)\geq ((1-\eta^{1/2})Q_--M\eta^{1/2})m(J)
	\end{equation}
	where $M$ is the uniform bound on eigenvalue flow speed for eigenvalues in $[a,b]$.
	
	On the other hand, we now bound $E(t_2;h)-E(t_1;h)$ above.
	To do this, we define $\tilde{E}(t;h)=E(t;h)-Bt$ and $\tilde{\mu}_m(t;h)=\mu_m(t;h)-Bt$
	where $B$ was the upper bound in \eqref{Bdefn}.
	
	Then the transformed quasi-eigenvalue windows $\tilde{\mu}_m(t;h)$ are non-increasing. From this it follows that if $\tilde{E}(s;h)\in [\tilde{\mu}_m(s;h)-h^{n+1},\tilde{\mu}_m(s;h)+h^{n+1}]$ and $m\in\mathcal{M}_h(s)$ for some $s\in J$, then $\tilde{E}(s';h)-\tilde{E}(s;h)<2h^{{n+1}}$, where $s'$ is the final time $t\in J$ such that $m\in\mathcal{M}_h(t)$ and $\tilde{E}(t;h)\in [\tilde{\mu}_m(t;h)-h^{n+1},\tilde{\mu}_m+h^{n+1}]$. This implies that $E(s';h)-E(s;h)<2h^{n+1}+B(s'-s)$.
	
	Generalising this idea, we can cover each $C_k(h)$ with a finite union of almost-disjoint intervals $\cup_jI_j$ with $I_j=[s_j,s_j']$ defined as follows:
	
	\begin{enumerate}
		\item We define $s_0:=\inf\{t\in J:E(t;h)\in W(t;h)\}$, and we choose an $m(0)\in\mathcal{M}_h(s_0)$ such 
		that $E(t;h)\in [\mu_{m(0)}(t;h)-h^{n+1},\mu_{m(0)}(t;h)+h^{n+1}]$ and $m(0)\in \mathcal{M}_h(t)$ for 
		all sufficiently small $t-s_0>0$.
		\item We then define $s_0':=\sup\{t\in J:E(t;h)\in [\mu_{m(0)}(t;h)-h^{n+1},\mu_{m(0)}(t;h)+h^{n+1}]\}$.
		\item If $\{t\in J: t>s_{j-1}' \textrm{ and } E(t;h)\in W(t;h)\}$ is empty, we terminate the inductive 
		process, otherwise we proceed inductively by defining $s_j:=\inf\{t\in J: t>s_{j-1}' \textrm{ and } E
		(t;h)\in W(t;h)\}$ and choosing a corresponding $m(j)\in\mathcal{M}_h(s_j)$ such that $E(t;h)\in [\mu_
		{m(j)}(t;h)-h^{n+1},\mu_{m(j)}(t;h)+h^{n+1}]$ and $m(j)\in \mathcal{M}_h(t)$ for all sufficiently 
		small $t-s_{j-1}>0$.
		\item We then define $s_j':=\sup\{t\in J:E(t;h)\in [\mu_{m(j)}(t;h)-h^{n+1},\mu_{m(j)}(t;h)+h^{n+1}]\}$.
	\end{enumerate}
	From the Weyl asymptotics, this procedure must terminate after finitely many iterations.
	
	\begin{remark}
		In the case that $E(t;h)$ is still in a quasi-eigenvalue window after the window corresponding to $\mu_{m(j)}$, we will have $s_{j+1}=s_j'$. This is the only kind of overlap possible between the intervals $I_j$. We also remark that the $m(j)$ are necessarily distinct, by the nature of this construction.
	\end{remark}
	For each such interval $I_j=[s_j,s_j']$, we have that $E(s_j';h)-E(s_j;h)\leq 2h^{n+1}+B(s_j'-s_j)$.
	
	As there can be at most $O(h^{-n})$ intervals $I_j$, we obtain:
	\begin{equation}
	\label{coolbound}
	\sum_{j}E(s_j';h)-E(s_j;h)\leq B\sum_j (s_j'-s_j)+O(h).
	\end{equation}
	For such eigenvalues, we thus obtain the upper bound
	\begin{eqnarray}
	\nonumber E(t_2;h)-E(t_1;h)&\leq&\nonumber \sum_j (E(s_j';h)-E(s_j;h))\\
	&+&\nonumber\left(m(J)(1-\eta^{1/2})-\sum_j (s_j'-s_j)\right)Q_++m(J)\eta^{1/2}M\\
	&\leq& \nonumber (B-Q_+)\sum_j (s_j'-s_j)+m(J)(1-\eta^{1/2})Q_++m(J)\eta^{1/2}M\\
	&\leq & \label{eigenjumpupper}(B-Q_+)m(C_j)+((1-\eta^{1/2})Q_++M\eta^{1/2})m(J)
	\end{eqnarray}
	in the limit $h\rightarrow 0$.
	Rearranging \eqref{eigenjumpupper} and using \eqref{eigenjumplower}, we arrive at
	\begin{equation}
	(Q_+-B)\frac{m(C_j)}{m(J)} \leq  2M\eta^{1/2}+(1-\eta^{1/2})(Q_+-Q_-).
	\end{equation}
    Hence by taking $\epsilon$ sufficiently small and then passing to sufficiently small $0<h<h_0(\epsilon)$ we can bound 
	\begin{equation}
	\frac{m(C_j)}{m(J)}
	\end{equation}
	by an arbitrarily small positive constant $\gamma$ for all $j\in\mathcal{F}$.
	
	Hence we have
	\begin{eqnarray}
	\int_J N(t;h)\, dt&\leq &\int_J \sum_{j\in \N}1_{C_j}\, dt\\
	 &\leq & \int_J \gamma\sum_{j\in \mathcal{F}} 1_{A_j}+\#(G\setminus \mathcal{F})\, dt\\
	&\leq & (\gamma \#\mathcal{F}+(\eta^{1/2}+O(\epsilon))\#G)m(J) \\
	&\leq & (\gamma+\eta^{1/2}+O(\epsilon))(\#G)m(J)
	\end{eqnarray}
	where we have used Proposition \ref{tildeGisfat} in the final line.
	
	Fixing sufficiently small $\epsilon>0$, for all $h<h_0(\epsilon)$ we have
	\begin{equation}
	\frac{1}{m(J)}\int_J \frac{N(t;h)}{\#\mathcal{M}_h(t)}\, dt \leq 1/4.
	\end{equation}
	It follows that for each such $h<h_0$, the set 
	\begin{equation}
	\{t\in J: \frac{N(t;h)}{\#\mathcal{M}_h(t)} \leq 1/2\}
	\end{equation}
	has measure at least $m(J)/2$.
	Taking a sequence $h_j\rightarrow 0$ and applying the Borel--Cantelli lemma completes the proof.
\end{proof}

We now prove an elementary spectral theory result that will show that the conclusion of Proposition \ref{specnonconckamprop} is in fact absurd, hence completing the proof of Theorem \ref{thm:main}.

We denote by $U$, the $h$-dependent span of all eigenfunctions with eigenvalues in $W(t;h)$.

\begin{prop}
	\label{spectralcont}
	For sufficiently small $h>0$, the projections
	\begin{equation}
	w_m(t_*,h)=\pi_U(v_m(t_*,h))
	\end{equation}
	are linearly independent.
\end{prop}
\begin{proof}
	First, we show that the estimate from Definition \ref{gevmodes} on the error of quasimodes implies that the projections $\pi_U(v_m(t_*,h))$ are large. In particular, for $m\in \mathcal{M}_h(t_*)$, we have
	\begin{eqnarray*}
		\left\|(\mathcal{P}_h(t_*)-\mu_m(t_*,h))\sum_{j\in\mathbb{N}}\langle v_m(t_*,h),u_j(t_*,h)\rangle u_j\right\|^2 & = & O(h^{2\gamma+2})\\ 
		\Rightarrow \sum_{|E_j-\mu_m|>h^{n+1}} |E_j(t,h)-\mu_m(t,h)|^2|\langle v_m(t_*,h),u_j(t_*,h)\rangle|^2 &=& O(h^{2\gamma+2}) \\
		\Rightarrow \pi_{U^\perp}(v_m(t_*,h)) &=& O(h^{\gamma-n}).
	\end{eqnarray*}
	Hence for sufficiently small $h$, we have 
	\begin{equation}
	\label{bigproj}
	\|w_m\|^2=\|\pi_U(v_m(t_*,h))\|^2= 1+O(h^{\gamma+1})+O(h^{2\gamma-2n}).
	\end{equation}
	It then follows from Definition \ref{gevmodes} and \eqref{bigproj} that the $w_m$ are almost orthogonal for distinct $m\in\mathcal{M}_h(t)$.
	\begin{eqnarray*}
		|\langle \pi_U(v_m(t_*,h)),\pi_U(v_k(t_*,h))\rangle|&\leq& |\langle v_m(t_*,h),v_k(t_*,h)\rangle|+|\langle\pi_{U^\perp}(v_m(t_*,h)),\pi_{U^\perp}(v_k(t,h))\rangle|\\
		& = & O(h^{\gamma+1})+O(h^{2\gamma-2n}).
	\end{eqnarray*}
	Hence there exist constants $c,C>0$ such that we have 
	\begin{equation}
	|\langle\pi_U(v_m(t_*,h)),\pi_U(v_k(t_*,h))\rangle-\delta_{k,m}|=O(h^{\gamma+1})+O(h^{2\gamma-2n}).
	\end{equation}
	for all sufficiently small $h$.
	
	If we enumerate the quasimodes $v_m(t_*,h)$ by positive integers rather than $m\in \Z^n$, we can then form the Gram matrix $M(h)\in \textrm{Mat}(\#\mathcal{M}_h(t_*),\R)$, with entries given by
	\begin{equation}
	M_{ij}(h)=\langle w_i,w_j\rangle.
	\end{equation}
	Since
	\begin{equation}
	\label{approxid2}
	\|M-I\|_{HS}= (\#\mathcal{M}_h(t_*))(O(h^{\gamma+1})+O(h^{2\gamma-2n}))=O(h^{\gamma+1-n})+O(h^{2\gamma-3n})
	\end{equation}
	we can invert $M=I+(M-I)$ as a Neumann series provided the exponents of the semiclassical parameter $h$ are positive. This can be ensured by taking $\gamma>3n/2$. Since $M$ is nonsingular, we can therefore conclude that the collection of functions
	\begin{equation}
	\{\pi_U(v_m(t_*,h)):m\in \mathcal{M}_h(t_*)\}
	\end{equation}
	are linearly independent.
\end{proof}


We are now in a position to complete the proof of Theorem \ref{nonqe}.

\begin{proof}[Completion of proof of Theorem \ref{nonqe}]
	Having fixed $\epsilon>0$ in Proposition \ref{quasispeedprop}, we have shown in Proposition \ref{specnonconckamprop} that there exists a $t_*\in (0,\delta)$ at which we have the spectral non-concentration result \eqref{specnonconckam} for a sequence $h_j\rightarrow 0$.
	
	On the other hand, we have shown in Proposition \ref{spectralcont} that the projections $\pi_U(v_m(t_*,h))$ are $\#\mathcal{M}_h(t_*)$ linearly independent vectors in a vector space of dimension $\dim(U)=N(t_*,h)<\#\mathcal{M}_h(t_*)/2$.	
	This contradiction completes the proof.
\end{proof}

\newpage

\section{Birkhoff normal form}
\label{sec:bnf}
\label{chp4}

In this section we construct a family of Birkhoff normal forms corresponding to a family of Gevrey smooth Hamiltonians $H(\theta,I;t)$, real-analytic in the parameter $t\in (-1,1)$. The introduction of this parameter leads to only minor changes in the argument of Popov \cite{popovkam}.

We formulate the KAM theorem from \cite{popovkam} in Section \ref{kamform} and outline the proof in Section \ref{kamproof}. We then complete the Birkhoff normal form construction following \cite{popovkam} in Section \ref{mainresultssec}

In Section \ref{leading}, we compute the leading order behaviour of this Birkhoff normal form as $t\rightarrow 0$, which was used in Proposition \ref{quasispeedprop} to obtain an expression for the derivatives of the quasi-eigenvalues of the operator $\mathcal{P}_h(t)$ constructed in Section \ref{sec:qbnf}.

\subsection{Notation}
\label{sec:notation}

We begin by introducing some notational conventions that will be used several times in this section.

\begin{defn}
	\label{domains1}
	If $D\subset \mathbb{R}^n$ and $s,r>0$ we write
	\begin{equation}
	\mathbb{T}^n+s :=\{z\in\mathbb{C}^n/2\pi\mathbb{Z}^n:|\textrm{Im}(z)|\leq s\}
	\end{equation}
	and
	\begin{equation}
	D_{s,r}:=\{\theta\in\mathbb{C}^n / 2\pi\mathbb{Z}^n:|\textrm{Im}(\theta)|<s\}\times \{I\in \mathbb{C}^n:|I|<r\},
	\end{equation}
	where $|\cdot|$ denotes the sup-norm on $\C^n$ induced by the $2$-dimensional $\ell^\infty$ norm on $\C$.
\end{defn}

These domains arise from considering the analytic extension of real analytic Hamiltonians in action-angle variables. In this topic it is common to bound derivatives of analytic functions using Cauchy estimates, which requires keeping track of shrinking sequences of domains. 

For simplicity of nomenclature, we call an analytic function of several complex variables real analytic if its restriction to a function of $n$ real variables is real-valued.

As a final notational convenience, we use $|\cdot|$ to denote the $\ell^1$ norm when applied to elements of $\Z^n$ throughout this paper, as well as the matrix norm induced by the sup norm on $\C^n$.

\subsection{Formulation of the KAM theorem}
\label{kamform}
Let $D^0\subset \mathbb{R}^n$ be a bounded domain, and consider a completely integrable Hamiltonian $H^0(I)=H^0(\theta,I):\mathbb{T}^n\times D^0 \rightarrow \mathbb{R}$ in action-angle coordinates. To begin with, we shall assume that this Hamiltonian is real analytic.

In addition, we assume the non-degeneracy condition $\det\left(\frac{\partial^2 H}{\partial I^2}\right)\neq 0$. This assumption implies that the map relating the action variable $I$ to the frequency $\omega=\nabla H^0(I)$ is locally invertible. In fact, we assume that 
\begin{equation}
\label{nondegen}
I\mapsto \nabla H^0(I)
\end{equation}
is a diffeomorphism from $D^0$ to $\Omega^0\subset \mathbb{R}^n$. The inverse to this map is given by $\nabla g^0$, where $g^0$ is the Legendre transform of $H^0$.

Taking $D\subset D^0$ a subdomain, and denoting by $\Omega=\nabla H^0(D)$ the corresponding frequency set, the phase space $\mathbb{T}^n\times D$ is then foliated by the family of Lagrangian tori $\{\mathbb{T}^n\times \{I\}:I\in D\}$ that are invariant under Hamiltonian flow associated to $H^0$.

The KAM theorem asserts that small perturbations of $H(\theta,I)=H^0(I)+H^1(\theta,I)$ on $\mathbb{T}^n\times D$ still possess a family of Lagrangian tori which fill up phase space up to a set of Liouville volume $o(1)$ in the size of the perturbation. More precisely, if $\Omega:=\{\omega:\omega=\nabla_I H^0\}$ is the set of frequencies for the quasi-periodic flow of $H^0$, the frequencies satisfying:
\begin{equation}
\label{diophantine}
|\langle\omega,k\rangle|\geq \frac{\kappa}{|k|^\tau}
\end{equation}
for all nonzero $k\in \mathbb{Z}^n$ and fixed $\kappa >0$ and $\tau>n-1$ also correspond to Lagrangian tori for the perturbed Hamiltonian $H$, provided $\|H-H^0\| < \epsilon(\kappa)$ in a suitable norm.
Such frequencies are said to be non-resonant, and we denote the set of non-resonant frequencies by $\Omega^*_\kappa$, suppressing the dependence on $\tau$ from our notation. These sets are obtained by taking the intersection of the sets
\begin{equation}
\{\omega\in \Omega:|\langle\omega,k\rangle|\geq \frac{\kappa}{|k|^\tau}\}
\end{equation}
over all nonzero $k\in \Z^n$, and hence $\cap_{\kappa>0}\Omega^*_\kappa$ is closed and perfect, with $\cap_{\kappa>0}\Omega^*_\kappa$ of full measure in $\Omega$, as can be seen from the observation that 
\begin{equation}
m(\{\omega\in \R^n: |\langle k,\omega\rangle|< \frac{\kappa}{|k|^\tau}\})=O(\frac{\kappa}{|k|^{\tau+1}}).
\end{equation}

We work with the sets
\begin{equation}
\Omega_\kappa:=\{\omega\in\Omega_\kappa^*:\textrm{dist}(\omega,\partial\Omega) \geq\kappa\}
\end{equation}
which have positive measure for sufficiently small $\kappa$.
It is also convenient to introduce notation for the set of points of Lebesgue density in $\Omega_\kappa$, which we denote by
\begin{equation}
\tilde{\Omega}_\kappa:=\{\omega\in \Omega:\frac{m(B(\omega,r)\cap \Omega_\kappa}{m(B(\omega,r))})\rightarrow 1 \textrm{ as }r\rightarrow 0\}.
\end{equation}
From the Lebesgue density theorem we have that $m(\tilde{\Omega}_\kappa)=m(\Omega_\kappa)$. We also note that a smooth function vanishing on $\Omega_\kappa$ is necessarily flat on $\tilde{\Omega}_{\kappa}$.

The construction of the Birkhoff normal form is a consequence of  Theorem \ref{kamwithparams}, which is a version of the KAM theorem localised around the frequency $\omega$ which is taken as an independent parameter. This version is particularly useful for the  Birkhoff normal form construction, as it makes it an easier task to check the regularity of the invariant tori with respect to the  frequency parameter.
To illustrate the setup of this theorem, we set 
\begin{equation}
\label{domains}
\Omega'=\{\omega\in \Omega: \textrm{dist}(\omega,\Omega_\kappa)\leq \kappa/2\},\quad D'=\nabla g^0(\Omega').
\end{equation}

Taking $z_0\in D'$ we let $I=z-z_0$ lie in a small ball of radius $R$ about $0$. That is, $R$ is chosen such that $B_R(z_0)\subset D$.
Taylor expanding gives us the expression
\begin{equation}
\label{taylorexpand}
H^0(z)=H^0(z_0)+\langle \nabla_z H^0(z_0),I\rangle+\int_0^1 (1-t)\langle \nabla_z^2 H^0(z_0+tI)I,I\rangle\, dt.
\end{equation}

We now take $\omega\in\Omega^0$ to be the corresponding frequency $\nabla H^0(z_0)$. The inverse of the frequency map is 
\begin{equation}
\label{psidef}
\psi_0(\omega)=\nabla g^0(\omega),
\end{equation}
where $g^0$ is the Legendre transform of $H^0$.
Hence we can write
\begin{equation}
H^0(z)=H^0(\psi_0(\omega))+\langle \omega, I\rangle +\langle P^0(I;\omega)I,I\rangle
\end{equation}
where $P^0$ is the quadratic remainder term in \eqref{taylorexpand}.
Expanding about the point $z_0=\nabla g^0(\omega)$, we can write our perturbation $H^1$ locally as
\begin{equation}
H^1(\theta,z)=H^1(\theta,\nabla g^0(\omega)+I)=P^1(\theta,I;\omega).
\end{equation}

This leads us to consider perturbed real analytic Hamiltonians in the form
\begin{equation}
\label{expnearfreq}
H(\theta,I;\omega)=H^0(\psi_0(\omega))+\langle\omega,I\rangle+P(\theta,I;\omega)=:N(I;\omega)+P(\theta,I;\omega).
\end{equation}
where
\begin{equation}
N(I;\omega)=H^0(\psi_0(\omega))+\langle \omega, I\rangle
\end{equation}
and
\begin{equation}
\label{defnofPzero}
P(\theta,I;\omega)=\langle P^0(I;\omega)I,I\rangle+P^1(\theta,I;\omega).
\end{equation}

The traditional formulations of the KAM theorem assert the existence of a Cantor family of tori that persist under small perturbations of a single Hamiltonian $H^0$ with domain $D$. In the framework laid out above, we now have a Cantor family of Hamiltonians parametrised by $\omega\in \Omega_\kappa$. Note that each of these Hamiltonians is only linear in $I$.

The essence the frequency localised KAM theorem in Theorem \ref{kamwithparams} is that for sufficiently small $P$, we can find a symplectic change of variables that transforms $H$ to a linear normal form in $I$ with remainder quadratic in $I$ for $\omega\in \Omega_\kappa$. This establishes the persistence of the Lagrangian torus  with frequency $\omega$. From Theorem \ref{kamwithparams}, one can obtain Theorem \ref{kamcons}, which establishes the existence of a Cantor family of invariant tori for the original Hamiltonian $H$ as with traditional formulations of the KAM theorem.

To work with Gevrey smooth Hamiltonians, we fix $L_2\geq L_0\geq 1$ and $A_0>1$, and assume that $H^0\in G^{\rho,1}_{L_0,L_2}(D^0\times (-1,1))$ and $g^0\in G^{\rho,1}_{L_0,L_2}(\Omega^0)$ with the estimates
\begin{equation}
\label{gevhamest}
\|H^0\|_{L_0,L_2},\|g^0\|_{L_0,L_2}\leq A_0.
\end{equation} 
For $L_2\geq L_1\geq 1$ we now consider the analytic family of Gevrey perturbations $${H^1\in G^{\rho,\rho,1}_{L_1,L_2,L_2}(\mathbb{T}^n\times D\times (-1,1))}$$
with the perturbation norm
\begin{equation}
\label{gevpertest}
\epsilon_H:=\kappa^{-2}\|H^1\|_{L_1,L_2,L_2}.
\end{equation}
The estimate \eqref{gevhamest} implies that there is a constant $C(n,\rho)$ dependent only on $n$ and $\rho$ such that taking
\begin{equation}
\label{Radbound}
R\leq \frac{C(n,\rho)\kappa}{A_0L_0^2}
\end{equation}
is sufficient to ensure that $B_R(z_0)\subset D$ for any $z_0\in D'$.

At this point we introduce the notational convention for this section that $C$ represents an arbitrary positive constant, dependent only on $n,\tau,\rho$ and $L_0$. Similarly, $c$ will represent a positive constant strictly less than $1$, also only dependent on $n,\tau,\rho$ and $L_0$. We will be explicit when we stray from this convention.

The  estimates \eqref{gevhamest} and \eqref{gevpertest}, together with Proposition A.3 in \cite{popovkam} show that our constructed functions $P^0$ and $P^1$ are in the Gevrey classes $${G^\rho_{CL_0,CL_2,CL_2}(B_R\times \overline{\Omega'}\times (-1,1))\subset G^\rho_{CL_2,CL_2}(B_R\times \overline{\Omega'}\times(-1,1))}$$ and $$G^{\rho,\rho,\rho,1}_{L_1,L_2,CL_2,L_2}(\mathbb{T}^n\times B_R\times \overline{\Omega'}\times (-1,1))$$ respectively, where the $C$ does not depend on $L_0$ or $L_2$.
Additionally we have the estimate
\begin{equation}
\|P^1\|_{L_1,CL_2,CL_2,CL_2}\leq \kappa^{-2}\epsilon_H.
\end{equation}
Dropping the factors in our Gevrey constants dependent only on $n,\tau,\rho,L_0$ for brevity of notation, we are in a position to state the local KAM theorem in terms of the weighted norm
\begin{equation}
\langle P\rangle_r:=r^2\|P^0\|_{L_2,L_2,L_2}+\|P^1\|_{L_1,L_2,L_2,L_2}
\end{equation}
for $0<r<R$.

\begin{thm}
	\label{kamwithparams}
	Suppose $0<\zeta\leq 1$ is fixed and $\kappa < L_2^{-1-\zeta}$. Then there exists ${N(n,\rho,\tau)>0}$ and $\epsilon>0$ independent of $\kappa,L_1,L_2,R,\Omega$ such that whenever the Hamiltonian 
	\begin{equation}
	H(\theta,I;\omega,t)=H^0(\psi_0(\omega);t)+\langle \omega,I\rangle+\langle P^0(I;\omega,t)I,I\rangle +P^1(\theta,I;\omega,t)
	\end{equation} 
	and $0<r<R$ are such that 
	\begin{equation}
	\label{kamwithparamasassump}
	\langle P\rangle_r <\epsilon\kappa r L_1^{-N}
	\end{equation}
	we can find $$\phi\in G^{\rho(\tau+1)+1,1}(\Omega\times (-3/4,3/4),\Omega)$$ and $$\Phi=(U,V)\in G^{\rho,\rho(\tau+1)+1,1}(\mathbb{T}^n\times \Omega\times (-3/4,3/4),\mathbb{T}^n\times B_R)$$ such that
	\begin{enumerate}
		\item For all $\omega \in \Omega_\kappa$ and all $t\in (-3/4,3/4)$, the map $\Phi_{\omega,t}=\Phi(\cdot;\omega,t):\mathbb{T}^n\rightarrow \mathbb{T}^n\times B_R$ is a $G^\rho$ embedding, with image $\Lambda_{\omega,t}$ an invariant Lagrangian torus with respect to the Hamiltonian $H_{\phi(\omega,t),t}(\theta,I)=H(\theta,I;\phi(\omega,t),t)$. The Hamiltonian vector field restricted to this torus is given by
		\begin{equation}
		X_{H_{\phi(\omega,t),t}}\circ \Phi_{\omega,t}=D\Phi_{\omega,t}\cdot \mathcal{L}_\omega
		\end{equation}
		where 
		\begin{equation}
		\mathcal{L}_\omega=\sum_{j=1}^n \omega_j\frac{\partial}{\partial \theta_j}\in T\mathbb{T}^n.\\
		\end{equation}
		
		\item There exist positive constants $A$ and $C$ dependent only on $n,\tau,\rho,L_0$ such that
		\begin{eqnarray}
		\label{kamwithparamsbound}
		& &|\partial_\theta^\alpha \partial_\omega^\beta (U(\theta;\omega,t)-\theta)|+r^{-1}|\partial_\theta^\alpha \partial_\omega^\beta V(\theta;\omega,t)|+\kappa^{-1}| \partial_\omega^\beta(\phi(\omega;t)-\omega)|\nonumber \\
		&\leq& A(CL_1)^{|\alpha|}(CL_1^{\tau+1}/\kappa)^{|\beta|}\alpha!^\rho\beta!^{\rho(\tau+1)+1}\frac{\langle P\rangle_r}{\kappa r}L_1^N
		\end{eqnarray}
		uniformly in $\mathbb{T}^n\times \Omega\times (-3/4,3/4)$.
	\end{enumerate}
	
	We remark that at the endpoint $t=0$, this result is trivial by taking $\phi(\omega,0)=\omega,U(\theta,\omega,0)=\theta$ and $V(\theta,\omega,0)=\nabla g^0(\omega)$.
\end{thm}

Theorem \ref{kamwithparams} can be proved in the same way as \cite{popovkam} Theorem 2.1, based on the rapidly converging iterative procedure introduced by Kolmogorov \cite{kolmogorov}. Indeed, much of the technicality in \cite{popovkam} involves the approximation of Gevrey class Hamiltonians by real analytic Hamiltonians. Thanks to the assumption of analyticity in $t$ in Theorem \ref{kamwithparams}, no such approximation is necessary in the $t$ parameter.

In the next section, we sketch the key steps in the proof of Theorem \ref{kamwithparams}, highlighting the points at which the presence of the $t$ parameter requires a modification of the argument in \cite{popovkam}.

First, we discuss the result that will comprise the steps of the iterative construction. Given a Hamiltonian in the form
\begin{equation}
H(\theta,I;\omega,t)=e(\omega;t)+\langle \omega,I \rangle +P(\theta,I;\omega,t)=N(I;\omega,t)+P(\theta,I;\omega,t),
\end{equation}
we aim to construct a $t$-dependent symplectomorphism $\Phi$ and a $t$-dependent frequency transformation $\phi$ such that for $\mathcal{F}=(\Phi,\phi)$, we have 
\begin{equation}
(H\circ \mathcal{F})(\theta,I;\omega,t)=N_+(I;\omega,t)+P_+(\theta,I;\omega,t)
\end{equation}
where $N_+(I,\omega,t) = e_+(\omega)+\langle I,\omega\rangle$ and with $|P_+|$ controlled by $|P|^r$ for some $r>1$.

This construction is analogous to that in \cite{poschel}.

\begin{thm}
	\label{kamstep}
	Suppose $\epsilon,h,v,s,r,\eta,\sigma,K$ are positive constants such that
	\begin{equation}
	\label{kamstepconsts}
	s,r<1,\;v<1/6,\;\eta<1/8,\;\sigma<s/5,\;\epsilon\leq c\kappa \eta r \sigma^{\tau+1},\;\epsilon\leq cvhr,\;h\leq \kappa/2K^{\tau+1}.
	\end{equation}
	where $c$ is a constant dependent only on $n$ and $\tau$.
	
	\noindent Suppose $H(\theta,I;\omega,t)=N(I;\omega,t)+P(\theta,I;\omega,t)$ is real analytic on $D_{s,r}\times O_h \times (-1,1)$, and $|P|_{s,r,h}\leq \epsilon$.
	Here, $D_{s,r}$ is as in Definition \ref{domains1} and 
	\begin{equation}
	O_h:=\{\omega\in \C^n:\textrm{dist}(\omega,\Omega_\kappa)<h\}.
	\end{equation}
	Then there exists a real analytic map \begin{equation}\mathcal{F}=(\Phi,\phi):D_{s-5\sigma,\eta r}\times O_{(1/2-3v)h}\times (-1,1)\rightarrow D_{s,r}\times O_h \end{equation} where the maps
	\begin{equation}
	\Phi:D_{s-5\sigma,\eta r}\times O_h\times (-1,1) \rightarrow D_{s,r}
	\end{equation}
	and
	\begin{equation}
	\phi: O_{(1/2-3v)h}\times (-1,1)\rightarrow O_h
	\end{equation}
	are such that 
	\begin{equation}
	H\circ\mathcal{F}=e_+(\omega,t)+\langle \omega,I\rangle+P_+(\theta,I;\omega,t)=N_+(I;\omega,t)+P_+(\theta,I;\omega,t)
	\end{equation}
	and we have the new remainder estimate
	\begin{equation}
	\label{newerrbd}
	|P_+|_{s-5\sigma,\eta r,(1/2-2v)h}\leq C\left(\frac{\epsilon^2}{\kappa r \sigma^{\tau+1}}+(\eta^2+K^ne^{-K\sigma})\epsilon\right).
	\end{equation}
	Moreover $\Phi$ is symplectic for each $(\omega,t)$ and has second component affine in $I$. 
	Finally, we have the following uniform estimates on the change of variables.
	\begin{equation}
	|W(\Phi-id)|,|W(D\Phi-Id)W^{-1}|\leq \frac{C\epsilon}{\kappa r \sigma^{\tau+1}}
	\end{equation}
	\begin{equation}
	|\phi-id|,vh|D\phi-Id|\leq \frac{C\epsilon}{r}
	\end{equation}
	where $W=\textrm{diag}(\sigma^{-1}Id,r^{-1}Id)$.
	All estimates are uniform in the analytic parameter $t\in (-1,1)$.
\end{thm}

This theorem is identical to \cite{popovkam} Proposition 3.2, with all estimates uniform in the parameter $t$. The proof is identical, with a detailed exposition in \cite{poschel}. The application of  \cite{popovkam} Lemma 3.4 to obtain the frequency transformation $\phi$ is replaced by Lemma \ref{popovlemma} in our setting. 

As in \cite{poschel},\cite{popov1}, Theorem \ref{kamstep} can be used to prove the KAM theorem for real analytic Hamiltonians $H(\theta,I;\omega,t)$. However, in order to treat the more general class of Gevrey smooth Hamiltonians $H\in G^{\rho,\rho,\rho,1}((\mathbb{T}^n\times D\times \Omega)\times (-1,1))$, we require the approximation result Proposition \ref{approxlemma}. 

\subsection{Proof of the KAM theorem}
\label{kamproof}
Following the proof of Theorem \ref{kamwithparams} in \cite{popovkam} Section 3, we extend the $P^j(\theta,I,\omega,t)$ to Gevrey functions 
\begin{equation}\tilde{P}^j\in G^{\rho,\rho,1}_{CL_1,CL_2,CL_2}(\mathbb{T}^n\times \mathbb{R}^{2n}\times (-1,1)) 
\end{equation}
where $C$ depends only on $n$ and $\rho$. We do this whilst preserving analyticity in $t$ by making use of an adapted version of the Whitney extension theorem for anisotropic Gevrey classes, from Proposition  \ref{whitneythm}.

We thus obtain the estimate
\begin{equation}
\label{whitneyapplication}
\|\tilde{P}^j\|\leq AL_1^{n+1}\|P^j\|
\end{equation}
where $A$ also only depends on $n$ and $\rho$.

We then cut-off $\tilde{P}^j$ without loss to have $(I,\omega)$ supported in $B_1\times B_{\bar{R}}\subset \R^{2n}$, where $1\ll\bar{R}$ is such that $\Omega^0\subset B_{\bar{R}-1}$. 
From here, we suppress the tilde in our notation, as well as the factor $C$ in our Gevrey constant.

We then require the following approximation result for functions in anisotropic Gevrey classes that plays a key role in the KAM iterative scheme. 

\begin{prop}
	\label{approxlemma}
	Suppose $P\in G^{\rho,\rho,1}_{L_1,L_2,L_2}(\mathbb{T}^n\times \R^{2n}\times (-1,1))$ satisfies ${\textrm{supp}_{(I,\omega)}(P)\subset B_1\times B_{\bar{R}}}$.
	If $u_j,w_j,v_j$ are positive real sequences monotonically tending to zero such that
	\begin{equation}
	\label{uvw}
	v_jL_2,w_jL_2\leq u_jL_1 \leq 1, \; v_0,w_0\leq L_2^{-1-\zeta} 
	\end{equation}
	where $1\leq L_1 \leq L_2$ and $0<\zeta \leq 1$ are fixed, then we can find a sequence of real analytic functions $P_j:U_j\rightarrow \mathbb{C}$ such that 
	\begin{equation}
	|P_{j+1}-P_j|_{U_{j+1}}\leq C(\bar{R}^n+1)L_1^n \exp\left(-\frac{3}{4}(\rho-1)(2L_1 u_j)^{-1/(\rho-1)}\right)\|P\|,
	\end{equation}
	\begin{equation}
	|P_0|_{U_0}\leq C(\bar{R}^n+1)\left(1+L_1^n \exp\left(-\frac{3}{4}(\rho-1)(2L_1 u_0)^{-1/(\rho-1)}\right)\right),
	\end{equation}
	and
	\begin{equation}
	|\partial_x^\alpha (P-P_j)(\theta,I;\omega,t)|\leq C(1+\bar{R}^n)L_1^nL_2 \exp\left(-\frac{3}{4}(\rho-1)(2L_1 u_j)^{-1/(\rho-1)}\right)
	\end{equation}
	in $\mathbb{T}^n\times B_1 \times B_{\bar{R}}\times (-1,1)$ for $|\alpha|\leq 1$,
	where 
	\begin{multline}
	U_j^m:=\{(\theta,I;\omega,t)\in \C^n / 2\pi \mathbb{Z}^n \times \C^n \times \C^n\times \C:\\
	|\textrm{Re}(\theta)|\leq\pi,|\textrm{Re}(I)|\leq 2,|\textrm{Re}(\omega)|\leq \bar{R}+1,|\textrm{Re}(t)|\leq 1,\\|\textrm{Im}(\theta)|\leq 2u_j,|\textrm{Im}(I)|\leq 2v_j,|\textrm{Im}(\omega_k)|\leq 2w_j,|\textrm{Im}(t)|\leq (2L_2)^{-1} \}
	\end{multline}
	and 
	\begin{equation}
	U_j:= U_j^1
	\end{equation}
	where we have identified $[-\pi,\pi]^n$ with $\T^n$ for simplicity of notation.
\end{prop}

The proof of Proposition \ref{approxlemma} can be found in  \cite{popovkam} Section 3. The first step is to extend $P$ to functions $F_j:U_j^2\rightarrow \mathbb{C}$ that are almost analytic in $(\theta,I,\omega)$ and are analytic in $t$. The Gevrey estimate on $t$-derivatives of $P$ imply that the Taylor expansions in $t$ have radius of convergence $L_2^{-1}$, and so the expression
	\begin{equation}
	\label{effjay}
	F_j(\theta+i\tilde{\theta},I+i\tilde{I},\omega+i\tilde{\omega},t+i\tilde{t}):=\sum_{\mathcal{M}_j} \partial_\theta^\alpha \partial_I^\beta \partial_\omega^\gamma P(\theta,I;\omega,t)\frac{(i\tilde{\theta})^\alpha (i\tilde{I})^\beta(i\tilde{\omega})^\gamma(i\tilde{t})^\delta}{\alpha!\beta!\gamma!\delta!}
	\end{equation}
	is convergent on $U_j^2$ where the index set is as in \cite{popovkam}.
	
	The remainder of the proof in \cite{popovkam} can be followed without change. As $P$ is analytic in $t$, we do not need to consider shrinking domains of analyticity as in the other variables.

The iterative scheme in \cite{popovkam} Section 3.3 can then be carried out, defining decreasing sequences of our parameters $s_j,r_j,h_j,\eta_j,\epsilon_j,\sigma_j,K_j$ such that the hypotheses of Theorem \ref{kamstep} are always satisfied, as well as decreasing sequences of the the parameters $u_j,v_j,w_j$ such that the hypotheses of the Proposition \ref{approxlemma} are always satisfied. Due to the modifications made in Theorem \ref{kamstep} and Proposition \ref{approxlemma} from their analogues in  \cite{popovkam}, all estimates are uniform in the analytic parameter $t\in (-1,1)$. 

Writing $U_j=U_j^1\cap \{|I|<r_j\}$ where $U_j^1$ is defined as in Proposition \ref{approxlemma} and applying Proposition \ref{approxlemma} to the terms $P^0,P^1$ from \eqref{defnofPzero}, we obtain sequences $P_j^0,P_j^1$ of real analytic functions in $U_j^1$ that are good approximations to $P^0$ and $P^1$.

Setting
\begin{equation}
P_j(\theta,I;\omega,t):=\langle P_j^0(I;\omega,t)I,I\rangle+P_j^1(\theta,I;\omega,t),
\end{equation}
Proposition \ref{approxlemma}, together with the factors picked up during the Whitney extension of $P^0,P^1$ in \eqref{whitneyapplication} imply the estimates
\begin{equation}
\label{Pzerobound}
|P_0|_{U_0}\leq \tilde{\epsilon}_0
\end{equation}
and
\begin{eqnarray}
\label{Pjbound}
|P_j-P_{j-1}|_{U_j}&\leq& \tilde{\epsilon}_j
\end{eqnarray}
where $\tilde{\epsilon}_j$ is a positive sequence rapidly converging to zero. 

Defining the Hamiltonian
\begin{equation}
H_j(\theta,I;\omega,t)=N_0(I;\omega)+P_j(\theta,I;\omega,t)=\langle \omega, I\rangle+P_j(\theta,I;\omega,t)
\end{equation}
which is real analytic in $U_j$, one can now perform the KAM iterative scheme as in \cite{popovkam} Proposition 3.5, using the key ingredient of Theorem \ref{kamstep}. For $j\geq 0$ we denote by $\mathcal{D}_j$ the class of real-analytic diffeomorphisms from $D_{j+1}\times O_{j+1}\times (-1,1)\rightarrow D_j\times O_j$ of the form
\begin{equation}
\mathcal{F}(\theta,I;\omega,t)=(\Phi(\theta,I;\omega,t),\phi(\omega;t))=(U(\theta;\omega,t),V(\theta,I;\omega,t),\phi(\omega;t))
\end{equation}
where $\Phi$ is affine in $I$ and canonical for fixed $(\omega,t)$. The domains are defined in terms of the parameters by $
D_j=D_{s_j,r_j}$ and $O_j=O_{h_j}$.

\begin{prop}
	\label{iterative}
	Suppose $P_j$ is real analytic on $U_j$ for each $j\geq 0$, and that we have the estimates
	\begin{equation}
	|P_0|_{U_0}\leq \tilde{\epsilon}_0
	\end{equation}
	and
	\begin{equation}
	|P_j-P_{j-1}|_{U_j}\leq \tilde{\epsilon}_j
	\end{equation}
	for each $j\geq 1$.
	
	Then for each $j\geq 0$, we can find a real-analytic normal form $N_j(I;\omega,t)=e_j(\omega,t)+\langle\omega,I\rangle$ and a real analytic map $\mathcal{F}^j$ given by
	\begin{equation}
	\mathcal{F}^{j+1}=\mathcal{F}_0\circ \ldots\circ \mathcal{F}_j: D_{j+1}\times O_{j+1}\times (-1,1) \rightarrow (D_0\times O_0)\cap U_j
	\end{equation}
	with the convention that the empty composition is the identity and where the $\mathcal{F}_j\in\mathcal{D}_j$ are such that
	\begin{equation}
	H_j\circ \mathcal{F}^{j+1}=N_{j+1}+R_{j+1}
	\end{equation}
	\begin{equation}
	|R_{j+1}|_{j+1}\leq \epsilon_{j+1}
	\end{equation}
	\begin{equation}
	\label{jacobian}
	|\bar{W}_j(\mathcal{F}_j-id)|_{j+1},|\bar{W}_j(D\mathcal{F}_j-Id)\bar{W}_j^{-1}|<\frac{C\epsilon_j}{r_j h_j}
	\end{equation}
	\begin{equation}
	\label{Sjbound}
	|\bar{W}_0(\mathcal{F}^{j+1}-\mathcal{F}^j)|_{j+1}<\frac{C\epsilon_j}{r_j h_j}
	\end{equation}
	where the constants $C$ depend only on $n$ and $\rho$ and $\bar{W}_j=\textrm{diag}(\sigma_j^{-1}\textrm{Id},r_j^{-1}\textrm{Id},h_j^{-1}\textrm{Id})$.
\end{prop}

To show that this iterative scheme converges in the Gevrey class $G^{\rho,\rho(\tau+1)+1,\rho(\tau+1)+1,1}$ requires Gevrey estimates for the $\mathcal{S}_j:=\mathcal{F}^{j+1}-\mathcal{F}^j$.
To this end we introduce the domains
\begin{equation}
\tilde{D}_j:=\{(\theta,I)\in D_j:|\textrm{Im}(\theta)|<s_j/2\},\; \tilde{O}_j:=\{\omega\in\mathbb{C}^n:\textrm{dist}(\omega,\Omega_\kappa)<h_j/2\}
\end{equation}
For multi-indices $\alpha,\beta$ with $|\beta|\leq m$, we also introduce the following notation for the $(m-|\beta|)$-th Taylor remainder in the frequency variable, centred at $\omega$.
\begin{equation}
R^m_\omega (\partial_\theta^\alpha\partial_\omega^\beta \mathcal{S}^j)(\theta,I,\omega',t):=\partial_\theta^\alpha\partial_\omega^\beta \mathcal{S}^j-\sum_{|\gamma|\leq m-|\beta|}(\omega'-\omega)^\gamma \partial_\theta^\alpha\partial_\omega^{\beta+\gamma} \mathcal{S}^j(\theta,I,\omega,t)/\gamma!.
\end{equation}
We then have the following Gevrey estimates from \cite{popovkam} Lemma 3.6 uniformly in the $t$ parameter.
\begin{lem}
	\label{gevest}
	\begin{equation}
	\label{Mgevbound}
	|\bar{W}_0\partial_\theta^\alpha \partial_\omega^\beta \mathcal{S}^j(\theta,0,\omega,t)|\leq \hat{\epsilon}AC^{|\alpha|+|\beta|}L_1^{|\alpha|+|\beta|(\tau+1)+1}\kappa^{-|\beta|}\alpha!^\rho \beta!^{\rho'}E_j^{1/2}
	\end{equation}
	for all $(\theta,0;\omega,t)\in \tilde{D}_{j+1}\times \tilde{O}_{j+1}\times (-1,1)$, where $\rho'=\rho(\tau+1)+1$.
	\begin{eqnarray}
	\label{Lgevbound}
	& &|\bar{W}_0 (R_\omega^m\partial_\theta^\alpha \partial_\omega^\beta \mathcal{S}^j)(\theta,0,\omega',t)|\\
	&\leq & \nonumber\hat{\epsilon}AC^{m+|\alpha|+1}L_1^{|\alpha|+(m+1)(\tau+1)+1}\kappa^{-m-1}\frac{|\omega-\omega'|^{m-|\beta|+1}}{(m-|\beta|+1)!}\alpha!^\rho(m+1)!^{\rho'}E_j^{1/2}
	\end{eqnarray}
	for all $\theta\in\mathbb{T}^n$, $\omega,\omega'\in\Omega_\kappa$ and $|\beta|\leq m$, where the constants $A,C$ only depend on $n,\rho,\tau,\zeta.$
\end{lem}
We can now bound derivatives in $t$, we use the Cauchy estimate from Proposition \ref{cauchyappendix}. This yields the following corollary.
\begin{cor}
	\label{gevestcor}
	\begin{equation}
	\label{Mgevcorbound}
	|\bar{W}_0\partial_\theta^\alpha \partial_\omega^\beta \partial_t^\gamma\mathcal{S}^j(\theta,0;\omega,t)|\leq \hat{\epsilon}AC^{|\alpha|+|\beta|+|\gamma|}L_1^{|\alpha|+|\beta|(\tau+1)+1}\kappa^{-|\beta|}\alpha!^\rho \beta!^{\rho'}\gamma!E_j^{1/2}
	\end{equation}
	for all $(\theta,0;\omega,t)\in \tilde{D}_{j+1}\times \tilde{O}_{j+1}\times (-3/4,3/4)$, where $\rho'=\rho(\tau+1)+1$.
	\begin{eqnarray}
	\label{Lgevcorbound}
	& &|\bar{W}_0 (R_\omega^m\partial_\theta^\alpha \partial_\omega^\beta\partial_t^\gamma \mathcal{S}^j)(\theta,0,\omega',t)|\\
	&\leq &\nonumber \hat{\epsilon}AC^{m+|\alpha|+|\gamma|+1}L_1^{|\alpha|+(m+1)(\tau+1)+1}\kappa^{-m-1}\frac{|\omega-\omega'|^{m-|\beta|+1}}{(m-|\beta|+1)!}\alpha!^\rho(m+1)!^{\rho'}\gamma!E_j^{1/2}
	\end{eqnarray}
	for all $\theta\in\mathbb{T}^n$,  $\omega,\omega'\in\Omega_\kappa$, $t\in (-3/4,3/4)$ and $|\beta|\leq m$, where the constants $A,C$ only depend on $n,\rho,\tau,\zeta.$
\end{cor}
From Proposition \ref{iterative} and Lemma \ref{gevestcor}, the rapid decay of $E_j$ implies that the limit
\begin{equation}
\label{jetnotation}
\partial_\theta^\alpha \partial_t^\gamma \mathcal{H}^\beta(\theta,\omega;t):=\lim_{j\rightarrow\infty} \partial_\theta^\alpha \partial_\omega^\beta\partial_t^\gamma(\mathcal{F}^j(\theta,0;\omega,t)-(\theta,0,\omega))
\end{equation}
exists for each $(\theta;\omega,t)\in \mathbb{T}^n\times \Omega_\kappa\times (-3/4,3/4)$, and each triple of multi-indices $\alpha,\beta,\gamma$. Convergence is uniform, and the limit is smooth in $\theta$ and $t$ and continuous in $\omega$, with $\partial_\theta^\alpha\partial_t^\gamma(\mathcal{H}^\beta)=\partial_\theta^\alpha\partial_t^\gamma \mathcal{H}^\beta$, justifying the notation in \eqref{jetnotation}. 

We now need to use the jet $\mathcal{H}=(\partial_\theta^\alpha\partial_t^\gamma \mathcal{H}^\beta)$ of continuous functions $\mathbb{T}^n\times \Omega_\kappa\times (-3/4,3/4) \rightarrow \mathbb{T}^n\times D \times \Omega$ to obtain a Gevrey function on $\mathbb{T}^n\times \Omega \times (-3/4,3/4)$ by using a Gevrey version of the Whitney extension theorem.
We define
\begin{equation}
(R_\omega^m \partial_\theta^\alpha \partial_t^\gamma \mathcal{H})_\beta(\theta,\omega',t):=\partial_\theta^\alpha\partial_t^\gamma \mathcal{H}^\beta(\theta,\omega',t)-\sum_{|\delta|\leq m-|\beta|}(\omega'-\omega)^\delta \partial_\theta^\alpha\partial_t^\gamma\mathcal{H}^{\beta+\delta}(\theta;\omega,t)/\gamma!
\end{equation}
In this notation, the results of Corollary \ref{gevestcor} yield
\begin{equation}
\label{impgevest}
|\bar{W}_0\partial_\theta^\alpha\partial_t^\gamma \mathcal{H}^\beta(\theta;\omega,t)|\leq \hat{\epsilon}AL_1(CL_1)^{|\alpha|}(CL_1^{\tau+1}/\kappa)^{|\beta|}C^\gamma\alpha!^{\rho}\beta!^{\rho'}\gamma!
\end{equation}
and
\begin{equation}
\label{impgevest2}
|\bar{W}_0(R_\omega^m \partial_\theta^\alpha \partial_t^\gamma\mathcal{H})_\beta(\theta,\omega',t)|\leq \hat{\epsilon}AL_1(CL_1)^{|\alpha|}(CL_1^{\tau+1}/\kappa)^{m+1}C^\gamma\frac{|\omega-\omega'|^{m-|\beta|+1}}{(m-|\beta|+1)!}\alpha!^{\rho}(m+1)!^{\rho'}\gamma!
\end{equation}
for $|\beta|\leq m$, and $(\theta,\omega,\omega',t)\in \mathbb{T}^n\times \Omega_\kappa \times \Omega_\kappa\times (-3/4,3/4)$, where $A$ and $C$ depend only on $n,\rho,\tau.$
These estimates allow us to apply the following consequence of Theorem \ref{whitneymainthm}.

\begin{prop}
	\label{whitneythm}
	Suppose $K\subset \mathbb{R}^n$ is compact, and $1\leq \rho < \rho'$. 
	If the jet $(f^{\alpha,\beta,\gamma})$ of functions $f^{\alpha,\beta,\gamma}:\mathbb{T}^n\times K\times (-3/4,3/4) \rightarrow \mathbb{R}$ is continuous on $\mathbb{T}^n\times K\times (-3/4,3/4)$ and is smooth in $(\theta,t)\in\T^n\times (-3/4,3/4)$ for each fixed $\omega\in K$ where
	\begin{equation}
	\partial_\theta^{\alpha'}\partial_t^{\gamma'}(f^{\alpha,\beta,\gamma})=f^{\alpha+\alpha',\beta,\gamma+\gamma'}
	\end{equation}
	and we have the estimates
	\begin{equation}
	| f^{\alpha,\beta,\gamma}(\theta;\omega,t)|\leq AC_1^{|\alpha|}C_2^{|\beta|}C_3^{|\gamma|}\alpha!^{\rho}\beta!^{\rho'}\gamma!
	\end{equation}
	and
	\begin{equation}
	|(R_\omega^m \partial_\theta^\alpha \partial_t^\gamma f)_\beta(\theta,\omega',t)|\leq AC_1^{|\alpha|}C_2^{m+1}C_3^{|\gamma|}\frac{|\omega-\omega'|^{m-|\beta|+1}}{(m-|\beta|+1)!}\alpha!^\rho (m+1)!^{\rho'}\gamma!
	\end{equation}
	then there exist positive constants $A_0,C_0$, dependent only on $(n,\rho,\tau)$ (in particular, independent of the set $K$) such that we can extend $f$ to $\tilde{f}\in G^{\rho,\rho',1}(\mathbb{T}^n\times \mathbb{R}^n\times (-3/4,3/4))$ such that $\partial_\theta^\alpha \partial_\omega^\beta\partial_t^\gamma \tilde{f}= f^{\alpha,\beta,\omega}$ on $\mathbb{T}^n\times K\times (-3/4,3/4)$ and
	\begin{equation}
	|\partial_{\theta}^\alpha \partial_\omega^\beta \partial_t^\gamma\tilde{f}(\theta,\omega)|\leq A_0 A\max(C_1,1)C_0^{|\alpha|+|\beta|+|\gamma|+n}C_1^{|\alpha|+n}C_2^{|\beta|}C_3^{|\gamma|}\alpha!^\rho \beta!^{\rho'}\gamma!
	\end{equation}
\end{prop}

The proof of Proposition \ref{whitneythm} is identical to that in \cite{popovkam} Theorem 3.7, making use of Theorem \ref{whitneymainthm} involving the parameter $t$.

Having established Proposition \ref{whitneythm}, the proof of Theorem \ref{kamwithparams} can be completed as in \cite{popovkam} Section 3.5 without modification.

\subsection{Birkhoff normal form}
\label{mainresultssec}
We obtain a Birkhoff normal form for near-integrable Hamiltonians using a version of the KAM theorem that is a consequence of Theorem \ref{kamwithparams}. The Gevrey index $\rho(\tau+1)+1$ frequently appears in these results, and so we introduce $\rho':=\rho(\tau+1)+1$.

\begin{thm}
	\label{kamcons}
	Fix $0<\zeta \leq 1$ and let $H^0(I;t)$ be a real-valued non-degenerate smooth family of Hamiltonian in $G^{\rho,1}(D^0\times (-1,1))$ and let $D$ be a subdomain with $\overline{D}\subset D^0$. We define $\Omega=\nabla H^0(D)$ and fix $L_2\geq L_1 \geq 1$ and $\kappa \leq L_2^{-1-\zeta}$ such that $L_2\geq L_0$ and $\Omega_\kappa \neq \emptyset$. Then there exists $N=N(n,\rho,\tau)$ and $\epsilon>0$ independent of $\kappa,L_1,L_2$ and $D\subset D^0$ such that for any $H\in G^{\rho,\rho,1}_{L_1,L_2,L_2}(\mathbb{T}^n\times D\times (-1,1))$ with norm
	\begin{equation}
	\label{epsilonhdef}
	\epsilon_H:=\kappa^{-2}\|H-H^0\|_{L_1,L_2,L_2}\leq \epsilon L_1^{-N}
	\end{equation}
	there exists a map 
	\begin{equation}
	\bar{\Phi}=(\bar{U},\bar{V})\in G^{\rho,\rho',1}(\mathbb{T}^n\times \Omega\times (-3/4,3/4),\mathbb{T}^n \times D)
	\end{equation}
	such that 
	\begin{enumerate}
		\item For each $\omega\in \Omega_\kappa$ and each $t\in (-3/4,3/4)$, $\Lambda_\omega=\{\bar{\Phi}(\theta;\omega,t):\theta\in\mathbb{T}^n\}$ is an embedded invariant Lagrangian torus of $H$, and $X_H\circ \bar{\Phi}(\cdot;\omega,t)=D\bar{\Phi}(\cdot;\omega,t)\cdot \mathcal{L}_\omega$.
		\item There exist constants $A,C>0$ independent of $\kappa,L_1,L_2$ and $D\subset D^0$ such that
		\begin{eqnarray}
		& &\nonumber|\partial_\theta^\alpha \partial_\omega^\beta(\bar{U}(\theta;\omega,t)-\theta)|+\kappa^{-1}|\partial_\theta^\alpha\partial_\omega^\beta(\bar{V}(\theta;\omega,t)-\nabla g^0(\omega))|\\
		&\leq &\label{kamconsbound} A(CL_1)^{|\alpha|}(CL_1^{\tau+1}/\kappa)^{|\beta|}\alpha!^\rho \beta!^{\rho'}L_1^{N/2}\epsilon_H^{1/2}
		\end{eqnarray}
		uniformly in $\mathbb{T}^n\times \Omega\times (-3/4,3/4)$.
	\end{enumerate}
\end{thm}
The proof of Theorem \ref{kamcons} is identical to \cite{popovkam} Theorem 1.1, making use of Theorem \ref{kamwithparams}. 

We can now use Theorem \ref{kamcons} to obtain the Birkhoff normal form as done in \cite{popovkam}. 
\begin{thm}
	\label{main1}
	Suppose the assumptions of Theorem \ref{kamcons} hold. Then there exists $N(n,\rho,\tau)>0$ and $\epsilon>0$ independent of $\kappa,L_1,L_2,D$ such that for any $H\in G^{\rho,\rho,1}_{L_1,L_2,L_2}(\mathbb{T}^n\times D\times (-1,1))$ with 
	\begin{equation}
	\label{pertsmall2}
	\epsilon_H\leq \epsilon L_1^{-N-2(\tau+2)}
	\end{equation}
	where $\epsilon_H$ is as in \eqref{epsilonhdef}, there is a family of $G^{\rho',\rho'}$ maps $\omega:D\times (-1/2,1,2)\rightarrow \Omega$ and a family of maps $\chi\in G^{\rho,\rho',\rho'}(\mathbb{T}^n\times D\times (-1/2,1,2),\mathbb{T}^n\times D)$ that are diffeomorphisms and exact symplectomorphisms respectively for each fixed $t\in (-1/2,1,2)$.
	Moreover, we can choose the maps $\omega$ and $\chi$ such that family of transformed Hamiltonians 
	\begin{equation}
	\tilde{H}(\theta,I;t):=(H\circ \chi)(\theta,I;t)
	\end{equation}
	is of Gevrey class $G^{\rho,\rho',\rho'}(\mathbb{T}^n\times D\times (-1/2,1,2))$ and can be decomposed as
	\begin{equation}
	\label{bnfeq}
	K(I;t)+R(\theta,I;t):=\tilde{H}(0,I;t)+(\tilde{H}(\theta,I;t)-\tilde{H}(0,I;t))
	\end{equation}
	such that:
	\begin{enumerate}
		\item $\mathbb{T}^n\times \{I\}$ is an invariant Lagrangian torus of $\tilde{H}(\cdot,\cdot;t)$ for each $I\in E_\kappa(t)=\omega^{-1}(\tilde{\Omega}_\kappa;t)$ and each $t\in (-1/2,1,2)$. 
		\item $
		\IB (\nabla K(I;t)-\omega(I;t))=\IB R(\theta,I;t)=0\quad\textrm{for all }(\theta,I;t)\in \T^n\times E_\kappa(t)\times (-1/2,1,2),\beta\in\N^n.
		$
		\item There exist $A,C>0$ independent of $\kappa,L_1,L_2,$ and $D\subset D^0$ such that we have the estimates
		\begin{eqnarray}
		& &|\TA\IB\TD \phi(\theta,I;t)|+|\IB\TD(\omega(I;t)-\nabla H^0(I;t))|+|\TA\IB\TD (\tilde{H}(\theta,I;t)-H^0(I;t))|\nonumber\\
		&\leq & \label{main1est}A\kappa C^{|\alpha|+|\beta|+|\delta|}L_1^{|\alpha|}(L_1^{\tau+1}/\kappa)^{|\beta|}\alpha!^\rho \beta!^{\rho'}\delta!^{\rho'}L_1^{N/2}\epsilon_H^{1/2}
		\end{eqnarray}
		uniformly in $\T^n\times D\times (-1/2,1,2)$ for all $\alpha,\beta$, where $\phi\in G^{\rho,\rho',\rho'}(\T^n\times D\times (-1/2,1,2))$ is such that $\langle \theta,I\rangle+\phi(\theta,I;t) $ generates the symplectomorphism $\chi$ in the sense of Proposition \ref{genfunc}.
	\end{enumerate}
\end{thm}

\begin{remark}
	For our purposes, high regularity in the $t$-parameter is not required, so we have dropped from analyticity to $G^{\rho'}$ regularity in $t$ at this point in order to simplify the proceeding arguments. I expect that analyticity in $t$ could be preserved by using a stronger variant of the Komatsu implicit function theorem than Corollary \ref{komatsucor}
\end{remark}

\begin{proof}
	We begin by taking $\epsilon,N$ as in Theorem \ref{kamcons} and noting that $\epsilon_H\leq \epsilon L_1^{-N-2}$ by assumption. 
	This implies that the factor $(ACL_1)L_1^{N/2}\sqrt{\epsilon_H}$ occurring in the Gevrey estimate \eqref{kamconsbound} can be bounded above by $AC\sqrt{\epsilon}$.
	Hence, taking $\epsilon$ small enough that both the conclusion to Theorem \ref{kamcons} holds as well as $AC\sqrt{\epsilon}<1/2$, we can first apply the Cauchy estimate from Proposition \ref{cauchyappendix} to \eqref{kamconsbound} in $t$, and then apply a variant of the Komatsu implicit function theorem, Corollary \ref{komatsucor}, to obtain a solution $\theta(\gamma;\omega,t):\T^n\times \Omega\times (-1/2,1,2)\rightarrow\T^n$ to the implicit equation
	\begin{equation}
	\bar{U}(\theta;\omega,t)=\gamma.
	\end{equation}
	Moreover, this solution satisfies the Gevrey estimate
	\begin{equation}
	|\partial_\gamma^\alpha\OB\TD(\theta(\gamma;\omega,t)-\gamma)|\leq AC^{|\alpha|+|\beta|+|\delta|}L_1^{|\alpha|}(L_1^{\tau+1}/\kappa)^{|\beta|}\alpha!^\rho\beta!^{\rho'}\delta!^{\rho'}L_1^{N/2}\sqrt{\epsilon_H}
	\end{equation}
	uniformly on $\T^n\times\Omega\times (-1/2,1,2)$.
	
	We set $F(\gamma;\omega,t):=\bar{V}(\theta(\gamma;\omega,t);\omega,t)$. In terms of $(\gamma;\omega,t)$, the Lagrangian torus $\Lambda_\omega$ is now given by $(\gamma,F(\gamma;\omega,t):\gamma\in\T^n)$ for each $\omega\in\Omega_\kappa$ and each $t\in (-1/2,1,2)$.
	Moreover, Proposition \ref{gevcomp2} on the composition of Gevrey functions gives us the estimate
	\begin{equation}
	\label{Fest}
	|\partial_\gamma^\alpha\OB\TD(F(\gamma;\omega,t)-\nabla g^0(\omega))|\leq A\kappa C^{|\alpha|+|\beta|+|\delta|}L_1^{|\alpha|}(L_1^{\tau+1}/\kappa)^{|\beta|}\alpha!^\rho\beta!^{\rho'}\delta!^{\rho'}L_1^{N/2}\sqrt{\epsilon_H}.
	\end{equation}
	We next construct functions $\psi\in G^{\rho,\rho',\rho'}(\R^n\times \Omega\times (-1/2,1,2))$ and $R\in G^{\rho',\rho'}(\Omega\times (-1/2,1,2))$ such that the function 
	\begin{equation}
	\label{Qdefn}
	Q(x;\omega,t):=\psi(x;\omega,t)-\langle x,R(\omega,t)\rangle
	\end{equation}
	is $2\pi$-periodic in $x$ and satisfies
	\begin{equation}
	\nabla_x\psi(x;\omega,t)=F(p(x),\omega,t)
	\end{equation}
	in $\R^n\times \Omega_\kappa\times (-1/2,1,2)$ where $p:\R^n\rightarrow \T^n$ is the canonical projection as well as the estimate
	\begin{eqnarray}
	\label{constructest}
	& &|\partial_x^\alpha\OB\TD Q(x;\omega,t)|+|\OB \TD (R(\omega,t)-\nabla g^0(\omega))|\\
	&\leq & A\kappa C^{|\alpha|+|\beta|+|\delta|}L_1^{|\alpha|}(L_1^{\tau+1}/\kappa)^{|\beta|}\alpha!^\rho\beta!^{\rho'}\delta!^{\rho'}L_1^{N/2}\sqrt{\epsilon_H}
	\end{eqnarray}
	for $(x;\omega,t)\in \R^n\times\Omega\times (-1/2,1,2)$.
	
	We do this by first integrating the canonical $1$-form $I\, dx$ over the chain
	\begin{equation}
	c_x:=\{(sx,F(p(sx);\omega,t)):0\leq s \leq 1\}\subset \R^n\times D.
	\end{equation}
	We define
	\begin{equation}
	\tilde{\psi}(x;\omega,t):=\int_{c_x}\sigma=\int_0^1 \langle F(p(sx);\omega,t),x\rangle \, ds
	\end{equation}
	in $\R^n\times \Omega\times (-1/2,1,2)$.
	From the estimate \eqref{Fest} it follows that $\tilde{\psi}(x;\omega,t)-\langle\nabla g^0(\omega),x\rangle$ is bounded above by the right hand side of \eqref{constructest} in $[0,4\pi]^n\times \Omega\times (-1/2,1,2)$. Hence if we define $R_j(\omega,t)=(2\pi)^{-1}\tilde{\psi}(2\pi e_j;\omega,t)$, then $R-\nabla g^0$ satisfies the required estimates in \eqref{constructest}.
	
	Since for $\omega\in\Omega_\kappa$ we know that $\Lambda_\omega$ is a Lagrangian torus, it follows that the  integral of the canonical $1$-form over any closed chain in $\Lambda_\omega$ is homotopy invariant. This means that such an integral is a homomorphism from the fundamental group of $\Lambda_\omega$ to $\R$. Hence
	\begin{equation}
	\tilde{\psi}(x+2\pi m;\omega,t)-\tilde{\psi}(x;\omega,t)=\langle 2\pi m,R(\omega,t)\rangle
	\end{equation}
	and so the function
	\begin{equation}
	\tilde{Q}(x;\omega,t):=\tilde{\psi}(x,\omega)-\langle x,R(\omega,t)\rangle
	\end{equation}
	both satisfies the Gevrey estimate in \eqref{constructest} and is $2\pi$-periodic in $x$ for $(\omega,t)\in \Omega_\kappa \times (-1/2,1,2)$.
	
	To obtain the sought $Q$ in \eqref{Qdefn} from $\tilde{Q}$, we use an averaging trick.
	Choosing $f\in G^{\rho}_C(\R^n)$ for some positive constant $C$ such that $f$ is supported in $[\pi/2,7\pi/2]^n$ and 
	\begin{equation}
	\sum_{k\in\mathbb{Z}^n}f(x+2\pi k)=1
	\end{equation}
	for each $x\in\R^n$, it then follows that
	\begin{equation}
	Q(x;\omega,t):=\sum_{k\in\mathbb{Z}^n}f(x+2\pi k)\tilde{Q}(x+2\pi k;\omega,t)
	\end{equation}
	is $2\pi$-periodic in $x$ for every $\omega\in\Omega$, and coincides with $\tilde{Q}$ for $\omega\in \Omega_\kappa$. Moreover, $Q$ satisfies the same Gevrey estimate \eqref{constructest} as $\tilde{Q}$. We define 
	\begin{equation}
	\psi(x;\omega,t):=Q(x;\omega,t)+\langle x, R(\omega,t)\rangle.
	\end{equation}
	Note that by multiplying $Q$ and $R-\nabla g^0$ by a cut-off function $h\in G^{\rho'}_{C/\kappa}$ which is equal to $1$ in a $\omega$-neighbourhood of $\Omega_\kappa$ and vanishes for $\textrm{dist}(\omega,\R^n\setminus\Omega)\leq \kappa/2$ where $C>0$ is independent of $\Omega\subset \Omega^0$, we can assume that $\psi(x;\omega,t)=\langle x, \nabla g^0(\omega)\rangle$ for $\textrm{dist}(\omega,\R^n\setminus\Omega)\leq \kappa/2$. This cutoff preserves the Gevrey estimates on $\psi$.
	
	Now since $\epsilon_HL_1^{N+2(\tau+2)}\leq \epsilon$, we have that $\kappa A(CL_1)(CL_1^{\tau+1}/\kappa)L_1^{N/2}\sqrt{\epsilon_H}\leq AC^2\sqrt{\epsilon}$. By taking $\epsilon$ sufficiently small we have that $\omega\mapsto \nabla_x\psi(x;\omega,t)$ is a diffeomorphism for any fixed $x\in \R^n$ from the Gevrey estimate \eqref{constructest}. Hence we have a $G^{\rho,\rho'}$-foliation of $\T^n\times D$ by Lagrangian tori $\Lambda_\omega=\{(p(x),\nabla_x\psi(x,\omega)):x\in\R^n\}$ where $\omega\in \Omega$.
	
	In the sought coordinate change, the action $I(\omega,t)$ of the Lagrangian torus $\Lambda_\omega$ will be  given by $R(\omega,t)$. Hence from \eqref{constructest} and Proposition \ref{komatsuprop}, it follows that for $\epsilon$ sufficiently small, the map 
	\begin{equation}
	(\omega,t)\mapsto (I(\omega,t),t)=(R(\omega,t),t)
	\end{equation}
	is a $G^{\rho',\rho'}$-diffeomorphism and we have the Gevrey estimate
	\begin{eqnarray}
	& &|\partial_I^\alpha\partial_t^\beta (\omega(I,t)-\nabla H^0(I;t))|\\
	&\leq & A\kappa C^{|\alpha|+|\beta|}(L_1^{\tau+1}/\kappa)^{|\alpha|}\alpha!^{\rho'}\beta!^{\rho'}L_1^{N/2}\sqrt{\epsilon_H}
	\end{eqnarray}
	uniformly for $(\theta,I,t)\in \T^n\times D\times (-1/2,1,2)$.
	
	We construct the sought symplectomorphism $\chi$ using the generating function $\Phi(x,I;t)$, setting
	\begin{equation}
	\Phi(x,I;t)=\psi(x,\omega(I;t);t)
	\end{equation}
	and noting that we have the required $2\pi$-periodicity of $\phi(x,I;t):=\Phi(x,I,t)-\langle x,I\rangle$, and from Proposition \ref{gevcomp2}, we also have the estimate
	\begin{equation}
	|\partial_x^\alpha \partial_I^\beta\TD(\Phi(x,I;t-\langle x,I\rangle))|\leq A\kappa C^{|\alpha|+|\beta|+|\delta|}L_1^{|\alpha|}(L_1^{\tau+1}/\kappa)^{|\beta|}\alpha!^\rho\beta!^{\rho'}\delta!^{\rho'}L_1^{N/2}\sqrt{\epsilon_H}.
	\end{equation}
	We can then apply Corollary \ref{komatsucor} to solve the implicit equation 
	\begin{equation}
	\partial_I\Phi(\gamma,I,t)=\theta
	\end{equation}
	for $\gamma$ with the estimate
	\begin{equation}
	|\partial_\theta^\alpha\partial_I^\beta\TD(\gamma(\theta,I,t)-\theta)|\leq  A\kappa C^{|\alpha|+|\beta|+|\delta|}L_1^{|\alpha|}(L_1^{\tau+1}/\kappa)^{|\beta|}\alpha!^\rho\beta!^{\rho'}\delta!^{\rho'}L_1^{N/2}\sqrt{\epsilon_H}.
	\end{equation}
	This completes the construction of a symplectomorphism $\chi$ satisfying
	\begin{equation}
	\chi(\partial_I\Phi(\theta,I,t),I)=(\theta,\partial_\theta \Phi(\theta,I,t)).
	\end{equation} 
	It follows that
	\begin{equation}
	(\theta,F(\theta;\omega,t))=\chi(\partial_I\Phi(\theta,I(\omega),t),I(\omega))=\chi(\theta,I(\omega),t)
	\end{equation}
	for $\omega\in\Omega_\kappa$ and so
	\begin{equation}
	\Lambda_\omega=\{\chi(\theta,I(\omega),t):\theta\in \T^n\}.
	\end{equation}
	for $(\omega,t)\in \Omega_\kappa\times (-1/2,1,2)$.
	
	We now set $\tilde{H},K,R$ as in the theorem statement in terms of the symplectomorphism $\chi$.
	Since $H$ is constant on $\Lambda_\omega$ for each $\omega\in \Omega_\kappa$, it follows that $R(\cdot,I;t)$ is identically zero for each $I=I(\omega)$ with $\omega\in\Omega_\kappa$. Hence $R$ is flat at $I\in E_\kappa(t)$, since each point in $E_\kappa(t)$ is of positive density in $I(\Omega_\kappa)$. 
	
	Finally, the Gevrey estimate in \eqref{main1est} for $\tilde{H}(\theta,I,t)-H(I,t)$ follows from Proposition \ref{gevcomp2}.
	This completes the proof.
\end{proof}

\subsection{Calculation of $\partial_t K_0(I,0)$}
\label{leading}
\label{derivcalcsec}
A crucial ingredient in the proof of Theorem \ref{thm:main} is the calculation of the derivative of quasi-eigenvalues in Proposition \ref{quasispeedprop} in the semiclassical limit $h\rightarrow 0$. From the truncated quantum Birkhoff normal form \ref{main2}, this can be reduced to the study of the $t$-dependence of the integrable term $K(I;t)$ in the classical Birkhoff normal form established in Theorem \ref{main1}. 

We now consider a $1$-parameter family of Hamiltonians $H(\theta,I;t)$ satsfying the assumptions of Theorem \ref{thm:main}. We can write
\begin{equation}
H(\theta,I;t)=H^0(I)+H^1(\theta,I;t)
\end{equation}
with
\begin{equation}
H^0(I):=H(\theta,I;0)
\end{equation}
and
\begin{equation}
\label{perttaylor}
H^1(\theta,I;t):=t\partial_tH(\theta,I;0)+\int_0^t (1-s)\partial_t^2 H(\theta,I;s)\, ds=t\partial_tH(\theta,I;0)+O(t^2)
\end{equation}
and we assume that $H$ additionally satisfies the assumptions of Theorem \ref{main1} with this choice of $H^0,H^1$. By applying two KAM stem iterations to $H(\theta,I;t)$, we obtain a transformed completely integrable component and reduce the order of magnitude of the $\theta$-dependent remainder.
An application of Theorem \ref{main1} to this transformed Hamiltonian produces a Birkhoff normal form, and \eqref{main1est}, yields an expression for $K(I;t)$ up to order $o(t)$.

The KAM step iterations required differ from that in Theorem \ref{kamstep}, in that they are not parametrised by $\omega\in\Omega$ and instead take place in action-angle space  $\T^n\times D$. Such a KAM step appears in the proof of the KAM theorem found in \cite{galavotti}. We first describe the KAM step without the presence of the parameter $t$ for simplicity. One begins with a perturbation
\begin{equation}
H(\theta,I)=H^0(I)+H^1(\theta,I)
\end{equation}
of a completely integrable Hamiltonian $H^0(I)$, and a fixed perturbation $H^1(\theta,I)$, both  analytic on the complex domain
\begin{equation}
\theta\in 2\pi \mathbb{C}^n\setminus 2\pi \mathbb{R}^n \quad |\textrm{Im}(\theta)|< s
\end{equation} 
\begin{equation}
\textrm{Re}(I)\in D\quad |\textrm{Re}(I)|< r.
\end{equation}
We assume that $\|H^1\|_{s,r}=O(\epsilon)$ in the uniform sense.

By consideration of the linearised Hamilton-Jacobi equation, we choose a symplectic transformation $\chi: \mathbb{T}^n\times D\rightarrow \mathbb{T}^n\times D$ with aim is to write
\begin{equation}
\label{aim}
\tilde{H}(\theta,I)=(H\circ\chi)(\theta,I)=\tilde{H}^0(I)+\tilde{H}^1(\theta,I)
\end{equation}
with $\tilde{H}^1=O(\epsilon^\alpha)$ for some $\alpha>1$. 
Then we have transformed a sufficiently small perturbation of an integrable Hamiltonian to an even smaller perturbation of a new integrable Hamiltonian, in a way we can hope to iterate.

Obtaining the ``new" error bound for $\tilde{H}^1$ necessarily requires a shrinking of the domains of analyticity, through the use of Cauchy estimates to control derivatives. Moreover, there is a more subtle shrinking of domain required in the $I$ variable, due to the infamous ``small-divisor" problem.
Specifically, $\chi$ is found using terms of the generating function
\begin{equation}
\label{genfunc}
\Phi(I',\theta)=i\sum_{k\in \mathbb{Z}^n:0<|k|\leq M} \frac{H^1_k(I')e^{ik\cdot \theta}}{\omega(I')\cdot k}. 
\end{equation}
where $H^1_k$ denotes the $k$-th Fourier coefficient of $H^1$, and $\omega=\nabla_I H^0(I)$ 
{(See \cite[(2.10)]{galavotti}).}

The denominators in \eqref{genfunc} can generally be zero, and so one must restrict to values of $I'$ for which we have a nonresonance condition
\begin{equation}
\omega(I')\cdot k\geq \frac{C}{|k|^2}
\end{equation}
for all $0<|k|\leq M$, where $C$ and $M$ are chosen suitably.
We also need to remove those actions $I'$ with $\textrm{dist}(I',\partial \Omega)\leq \tilde{\rho}$ so that the perturbed tori do not escape the coordinate patch. (See \cite[(3.12)]{galavotti} for the choice of the constant $\tilde{\rho}$).
This leads to the definition of the set
\begin{equation}
\tilde{D}_1=\{I\in D:\textrm{dist}(I,\partial D)>\tilde{\rho}\textrm{ and } \omega(I)\cdot k \geq \frac{C}{|k|^2}\textrm{ for all } 0<|k|\leq M.\}
\end{equation}
For any $\tilde{I}\in \tilde{D}_1$ the expression \eqref{genfunc} is certainly defined, but as the domain might have rather rough boundary, it is convenient to slightly enlarge $\tilde{D}_1$ to the open set
\begin{equation}
D_1=\cup_{I\in \tilde{D}_1} B(I,\tilde{\rho}/2).
\end{equation}
Upon restricting to this action set for suitable $C$ and $M$, the objective of \eqref{aim} can indeed be achieved, and the ``integrable part" of the new Hamiltonian can be written as
\begin{equation}
\tilde{H}^0(I)=H^0(I)+(2\pi)^{-n}\int H^1(\theta,I)\, d\theta
\end{equation}
(See \cite[(3.38)]{galavotti}).
The overall transformed Hamiltonian is then given by
\begin{equation}
\tilde{H}(\tilde{\theta},\tilde{I})=\tilde{H}^0(\tilde{I})+\tilde{H}^1(\tilde{\theta},\tilde{I})
\end{equation}
in the domain $\T^n\times D_1$ with 
\begin{equation}
\label{3/2}
\|\tilde{H}^1\|=O(\epsilon^{3/2})
\end{equation}.

The classical KAM theorem is then proven in \cite{galavotti} by iterating this procedure, carefully choosing the $C,M,\tilde{\rho}$ and the analyticity parameters $r,s$ so that the estimate \eqref{3/2} is satisfied with every step, ensuring convergence, and so that the limiting domain $\cap_j D_j$ of nonresonant actions is of large measure. A full discussion of this procedure can be found in \cite{galavotti}.

We now return to our setting of the one-parameter family of Hamiltonians
$$H(\theta,I;t)=H^0(I)+H^1(\theta,I;t)$$
One iteration of the KAM step outlined above yields a family of symplectomorphisms
\begin{equation}
\chi_1:\T^n\times D_1\rightarrow \T^n\times D
\end{equation}
parametrised by $t$ such that
\begin{equation}
\label{firstiter}
\tilde{H}(\theta,I;t)=(H\circ \chi_1)(\theta,I;t)=H^0(I)+t\cdot (2\pi)^{-n}\int_{\T^n}\partial_tH(\theta,I;0)\, d\theta+\tilde{H}^1(\theta,I;t)
\end{equation}
where the second term comes from \eqref{perttaylor} and the error term $\tilde{H}^1(\theta,I;t)=O(t^{3/2})$.
Regarding this transformed Hamiltonian as being a small  perturbation of the integrable Hamiltonian
\begin{equation}
\tilde{H}^0(I;t)=H^0(I)+t\cdot (2\pi)^{-n}\int_{\T^n}\partial_t H(\theta,I;0)\, d\theta
\end{equation}
we perform one more KAM iteration to obtain another family of symplectomorphisms
\begin{equation}
\chi_2:\T^n\times D_3\rightarrow \T^n\times D_2
\end{equation}
parametrised by $t$ such that
\begin{eqnarray}
\nonumber\tilde{\tilde{H}}(\theta,I;t)&=&(\tilde{H}\circ\chi_2)(\theta,I;t)\\
\nonumber &=&H^0(I)+t\cdot (2\pi)^{-n}\int_{\T^n}\partial_t H(\theta,I;0)\, d\theta +(2\pi)^{-n}\int_{\T^n}\tilde{H}^1(\theta,I;t)\, d\theta +\tilde{\tilde{H}}^1(\theta,I;t).
\end{eqnarray}
Moreover, by taking our initial choice of nonresonance parameter $C$ sufficiently small, we can ensure that the action domain $D_3$ contains a collection of nonresonant actions $E_\kappa(t)$ with
\begin{equation}
\nabla_I(\tilde{\tilde{H}}^0(E_\kappa(t)))=\Omega_\kappa
\end{equation}
where 
\begin{equation}
\tilde{\tilde{H}}^0(I;t)=H^0(I)+t\cdot (2\pi)^{-n}\int_{\T^n}H^1(\theta,I)\, d\theta+(2\pi)^{-n}\int_{\T^n}\tilde{H}^1(\theta,I;t)\, d\theta.
\end{equation}

We now summarise the preceding discussion.
\begin{prop}
	\label{bnf15}
	Suppose $H(\theta,I;t)$ is a family of real analytic perturbations of the completely integrable non-degenerate Hamiltonian $H^0(I)$ in $\T^n\times D\times (-1,1)$ that has an analytic extension to
	\begin{equation}
	W_{s,r}(D):= \{(\theta,I)\in \C^n/2\pi\Z \times \C^n:|\textrm{Im}(\theta)|<s,\textrm{dist}(I,D)<r \}.
	\end{equation}
	
	Suppose further that the conditions
	\begin{equation}
	\left|\frac{\partial H^0}{\partial I}\right|\leq E,
	\end{equation}
	\begin{equation}
	\left|\left(\frac{\partial^2 H^0}{\partial I^2}\right)^{-1}\right|\leq \eta,
	\end{equation}
	and
	\begin{equation}
	\left(\left|\frac{\partial H^1}{\partial I}\right|+r^{-1}\left|\frac{\partial H^1}{\partial \theta}\right|\right)\leq \epsilon
	\end{equation}
	are satisfied. 
	
	Then for sufficiently small $\delta > 0$, there exists a subdomain $\tilde{D}\subset D$ and a family of real analytic symplectic maps
	\begin{equation}
	\chi:\T^n\times \tilde{D} \times (-\delta,\delta) \rightarrow \T^n\times D
	\end{equation}
	that analytically extend to a new domain of holomorphy
	\begin{equation}
	W_{s_+,r_+}(\tilde{D})
	\end{equation} 
	such that
	\begin{equation}
	(H\circ \chi)(\theta,I;t)=\tilde{H}^0(I;t)+ \tilde{H}^1(\theta,I;t).
	\end{equation}
	with
	\begin{equation}
	\partial_t\tilde{H}^0(I;0)=(2\pi)^{-n}\int_{\T^n} \partial_t H(\theta,I;0)\, d\theta
	\end{equation} 
	and
	\begin{equation}
    \|\tilde{H}^1\|_{s_+,r_+} = O(t^{9/4})
	\end{equation}
	with constant depending only on $n$ and $E$.
	Moreover, this domain $\tilde{D}$ contains a collection $E_\kappa(t)$ of actions such that 
	\begin{equation}
	\nabla_I(\tilde{H}^0)(E_\kappa(t))=\Omega_\kappa.
	\end{equation}
\end{prop}
We can also generalise this result to the Gevrey setting.
\begin{prop}
	\label{bnf175}
	Suppose $H(\theta,I;t)\in G^{\rho,\rho,1}(\T^n\times D\times (-1,1)$ is a family of Hamiltonians satisfying the assumptions of Theorem \ref{main1} where $H^0(I):=H(\theta,I;0)$ for fixed $\rho>1$, and choose $\kappa>0$ small. Then for sufficiently small $\|H(\theta,I;t)-H^0(I)\|_{L_1,L_2,L_2}$, there exists a subdomain $\tilde{D}\subset D$ and a $G^{\rho,\rho,1}$ family of  symplectic maps
	\begin{equation}
	\chi:\T^n\times \tilde{D} \times (-1,1) \rightarrow \T^n\times D
	\end{equation}
	such that
	\begin{equation}
	(H\circ \chi)(\theta,I;t)=\tilde{H}^0(I;t)+ \tilde{H}^1(\theta,I;t).
	\end{equation}
	with
	\begin{equation}
	\label{newh0}
	\partial_t\tilde{H}^0(I;0)=(2\pi)^{-n}\int_{\T^n} \partial_t H(\theta,I;0)\, d\theta
	\end{equation} 
	and
	\begin{equation}
	\label{why9on4}
	\|\tilde{H}^1\|_{CL_1,CL_2,CL_2} = O(t^{9/4})
	\end{equation}
	with constant independent of $\kappa$ and with $C$ dependent only on $n$ and $\rho$.
	
	Moreover, the domain $\tilde{D}$ contains $E_\kappa(t)=\omega^{-1}(\Omega_\kappa;t)=(\nabla_I \tilde{H}^0)^{-1}(\Omega_\kappa;t)$
\end{prop}
\begin{proof}
	This result is established via the approximation of Gevrey functions by real-analytic functions. First, we define
	\begin{equation}
	H^0(I)=H(\theta,I;0)
	\end{equation}
	and 
	\begin{equation}
	H^1(\theta,I;t)=H(\theta,I;t)-H(\theta,I;0)=\int_0^t \partial_tH(\theta,I;s)\, ds
	\end{equation}
	and	use Proposition \ref{whitneythm} to boundedly extend $H^0$ and $H^1$ to the domain $\T^n\times \R^n\times (-1,1)$, before cutting off in $I$ to a ball $B_{\tilde{R}}$ with $D\subset B_{\tilde{R}-1}$.
	From the same methods used in the proof of Proposition \ref{approxlemma}, we may then construct sequences of real analytic functions $P^0_j$ and $P^1_j$ on shrinking $j$ dependent complex domains $U_j$ containing $\T^n\times \R^n\times (-1,1)$ with a corresponding sequence $u_j\rightarrow 0$ such that 
	\begin{equation}
	\label{pkuj}
	|P^k_{j+1}-P^k_j|_{U_{j+1}}\leq C(D^0,L_1,L_2) \exp\left(-\frac{3}{4}(\rho-1)(2L_1 u_j)^{-1/(\rho-1)}\right)\|H^k\|
	\end{equation}
	and
	\begin{equation}
	|\partial_x^\alpha (P^k_j-H^k)(\theta,I;t)|\leq C(D^0,L_1,L_2) \exp\left(-\frac{3}{4}(\rho-1)(2L_1 u_j)^{-1/(\rho-1)}\right)
	\end{equation}
	in $\mathbb{T}^n\times B_{\tilde{R}}\times (-1,1)$ for $|\alpha|\leq 1$. These sequences $P^k_j$ are convergent in ${G^{\rho,\rho,1}(\T^n\times\R^n\times (-1,1))}$, as is shown in \cite{gevapproxref} Proposition 2.2. (This fact can be readily obtained by applying Cauchy estimates to \eqref{pkuj}.)
	
	Now for each $j\in \N$, we can carry out the first KAM step for the real analytic Hamiltonian $P_j=P^0_j+P^1_j$ to obtain a real analytic symplectic map 
	\begin{equation}
	\chi_j:\T^n\times D_1 \rightarrow \T^n\times D
	\end{equation}
	defined in shrinking holomorphy domains such that
	\begin{equation}
	((P^0_j+P^1_j)\circ \chi_j)(\theta,I;t)=P^0(I)+t\cdot (2\pi)^{-n} \int_{\T^n} \partial_t P_j^1(\theta,I;0)\, d\theta+\tilde{P}^1(\theta,I;t)
	\end{equation}
	with $\|P_j^1\|=O(t^{3/2})$.
	Note that for an individual KAM step, the symplectic map $\chi_j$ is defined using a generating function $\Phi_j$ that is a weighted sum of finitely many Fourier components of $P^1_j$ (see \eqref{genfunc} and \cite[Equation 3.14]{galavotti}.)
	This implies that as $P^0_j+P^1_j\rightarrow H^0+H^1$ in $G^{\rho,\rho,1}(\T^n\times D_1\times (-1,1))$, the generating functions $\Phi_j$ converges to some
	\begin{equation}
	\Phi\in G^{\rho,\rho,1}(\T^n\times D_1\times (-1,1))
	\end{equation}
	in the $G^{\rho,\rho,1}$ sense.
	From Corollary \ref{komatsucor}, it follows that the corresponding symplectic maps $\chi_j$ converge to some
	\begin{equation}
	\chi^1\in G^{\rho,\rho,1}(\T^n\times D_1\times (-1,1))
	\end{equation} 
	in the Gevrey sense.
	
	Similarly, the symplectic maps $\tilde{\chi}_j$ that  comprise a single KAM step for the Hamiltonians
	\begin{equation}
	(P^0_j+P^1_j)\circ\chi_j
	\end{equation}
	can also be seen to converge to some
	\begin{equation}
	\chi^2\in G^{\rho,\rho,1}(\T^n\times D_2,\T^n\times D_1).
	\end{equation}
	
	It follows that the family of symplectic maps $\chi_j\circ\tilde{\chi}_j$ whose existence is asserted by applying Proposition \ref{bnf15} to $P^0_j+P^1_j$ converge to some $\chi:=\chi^1\circ\chi^2$ in the $G^{\rho,\rho,1}$-sense.
	Moreover, if we write 
	\begin{equation}
	(P^0_j+P^1_j)\circ \chi_j\circ \tilde{\chi}_j=\tilde{H}^0_j(I;t)+ \tilde{H}^1_j(\theta,I;t).
	\end{equation}
	in the notation of Proposition \ref{bnf15}, we have that $\tilde{H}^k_j$ are convergent sequences in $G^{\rho,\rho,1}$, and so it follows that their limits $\tilde{H}^0,\tilde{H}^1$ satisfy
	\begin{equation}
	\partial_t\tilde{H}^0(I;0)=(2\pi)^{-n}\int_{\T^n} \partial_t H(\theta,I;0)\, d\theta
	\end{equation}
	and
	\begin{equation}
	\label{estimate}
	\|\tilde{H}^1\|_{CL_1,CL_2,CL_2} = O(t^{9/4})
	\end{equation}
	as required.
\end{proof}

Finally, we complete our computation of $\partial_tK_0(I;0)$ for a given Hamiltonian $H(\theta,I;t)$ satisfying the conditions of Theorem \ref{main1} by applying Proposition \ref{bnf175} to $H$, prior to applying Theorem \ref{main2} to compute the Birkhoff normal form of the transformed Hamiltonian $\tilde{H}(\theta,I;t)$.

By applying Proposition \ref{bnf175} to $H(\theta,I;t)$ with $\|H(\theta,I;t)-H(\theta,I;0)\|$ sufficiently small, we can then apply Theorem \ref{main1} to the Hamiltonian 
\begin{equation}
\tilde{H}(\theta,I;t)=\tilde{H}^0(I;t)+\tilde{H}^1(\theta,I;t)
\end{equation}
with an improved error term.

\begin{prop}
	\label{bnf2}
	Suppose the assumptions of Theorem \ref{kamcons} hold for the Hamiltonian
	\begin{equation}
	H(\theta,I;t)\in G^{\rho,\rho,1}(\mathbb{T}^n\times D\times (-1,1)).
	\end{equation} 
	Then there exists $N(n,\rho,\tau)>0$ and $\epsilon>0$ independent of $L_1,L_2,D$ such that for any $H\in G^{\rho,\rho,1}_{L_1,L_2,L_2}(\mathbb{T}^n\times D\times (-1,1))$ with 
	\begin{equation}
	\label{pertsmall3}
	\kappa^{-2}\|H(\theta,I;t)-H(\theta,I;0)\|_{L_1,L_2,L_2}=\epsilon_H\leq \epsilon L_1^{-N-2(\tau+2)}
	\end{equation}
	there is a subdomain $\tilde{D}\subset D$ containing $E_\kappa(0)$ and a family of $G^{\rho',\rho'}$ maps $\omega:\tilde{D}\times (-1/2,1,2)\rightarrow \Omega$ and a family of maps $\chi\in G^{\rho,\rho',\rho'}(\mathbb{T}^n\times \tilde{D}\times (-1/2,1,2),\mathbb{T}^n\times \tilde{D})$ that are diffeomorphisms and exact symplectomorphisms respectively for each fixed $t\in (-1/2,1,2)$.
	Moreover, we can choose the maps $\omega$ and $\chi$ such that family of transformed Hamiltonians 
	\begin{equation}
	\tilde{H}(\theta,I;t):=(H\circ \chi)(\theta,I;t)
	\end{equation}
	is of Gevrey class $G^{\rho,\rho',\rho'}(\mathbb{T}^n\times \tilde{D}\times (-1/2,1,2))$ and can be decomposed as
	\begin{equation}
	K(I;t)+R(\theta,I;t):=\tilde{H}(0,I;t)+(\tilde{H}(\theta,I;t)-\tilde{H}(0,I;t))
	\end{equation}
	such that:
	\begin{enumerate}
		\item $\mathbb{T}^n\times \{I\}$ is an invariant Lagrangian torus of $\tilde{H}(\cdot,\cdot;t)$ for each $I\in E_\kappa(t)=\omega^{-1}(\tilde{\Omega}_\kappa)$ and each $t\in (-1/2,1,2)$. 
		\item $
		\IB (\nabla K(I;t)-\omega(I;t))=\IB R(\theta,I;t)=0\quad\textrm{for all }(\theta,I;t)\in \T^n\times E_\kappa(t)\times (-1/2,1,2),\beta\in\N^n.
		$
		\item There exist $A,C>0$ independent of $\kappa,L_1,L_2,$ and $D\subset D^0$ such that we have the estimates
		\begin{eqnarray}
		& &|\TA\IB\TD \phi(\theta,I;t)|+|\IB\TD(\omega(I;t)-\nabla \tilde{H}^0(I;t))|+|\TA\IB\TD (\tilde{H}(\theta,I;t)-\tilde{H}^0(I;t))|\nonumber\\
		&\leq & \label{main1est2}A C^{|\alpha|+|\beta|+|\delta|}L_1^{|\alpha|}(L_1^{\tau+1}/\kappa)^{|\beta|}\alpha!^\rho \beta!^{\rho'}\delta!^{\rho'}L_1^{N/2}|t|^{9/8}
		\end{eqnarray}
		uniformly in $\T^n\times \tilde{D}\times (-1/2,1,2)$ for all $\alpha,\beta$, where $\phi\in G^{\rho,\rho',\rho'}(\T^n\times \tilde{D}\times (-1/2,1,2))$ is such that $\langle \theta,I\rangle+\phi(\theta,I;t) $ generates the symplectomorphisms $\chi$ in the sense of Proposition \ref{genfunc} and $\tilde{H}^0,\tilde{H}^1$ are as in Proposition \ref{bnf175}.
		
		\item \begin{equation}
		\label{chp4tderiv}
		\partial_t K(I;t)=(2\pi)^{-n}\int_{\T^n} \partial_tH(\theta,I;0)+o(1)
		\end{equation}
		uniformly in $\T^n\times \tilde{D}\times (-1/2,1,2)$.
	\end{enumerate}
\end{prop}
\begin{proof}
	The only new claim in this Proposition is \eqref{chp4tderiv}, which follows from \eqref{main1est2} and the expression \eqref{newh0} for $\tilde{H}^0$. Note that the exponent $9/8$ in \eqref{main1est2} comes from \eqref{why9on4} and the square root in \eqref{main1est}.
\end{proof}

\section{Quantum Birkhoff normal form}
\label{sec:qbnf}

Through the work in section \ref{sec:bnf}, we have now established that the Birkhoff normal form construction in \cite{popovkam} preserves smoothness in the $t$ parameter when applied to the Hamiltonian $P_0(x,\xi;t)$ that is the principal symbol of the operator introduced in \eqref{operatorP}.
This regularity in $t$ propagates through the quantum Birkhoff normal form construction in \cite{popovquasis}, which we discuss in this section.
The upshot of this regularity in $t$ is that the quasimodes constructed in \cite{popovquasis} Section 2.4 can be chosen to have associated quasi-eigenvalues varying smoothly in the parameter $t$. We discuss these quasimodes in Section \ref{gevquasisec}.
\subsection{Quantum Birkhoff normal form}

In \cite{popovquasis}, a quantum Birkhoff normal form is constructed for semiclassical pseudodifferential operators of the form \eqref{operatorP} after first obtaining a classical Birkhoff normal form for the principal symbol of regularity $G^{\rho,\rho'}$ as in Theorem \ref{main1}.
This normal form uses the Gevrey symbol classes introduced in Section \ref{gevsymbsec} and is stated in Theorem \ref{main2}.
We remark that the proof is presented in \cite{popovquasis} for differential operators, but can be carried out without change if the $\mathcal{P}_h$ is a pseudodifferential operator.

We denote by $\chi_1$ the symplectomorphism that transforms the completely integrable Hamiltonian $P(x,\xi;0)$ into action-angle coordinates $H=P\circ(\chi_1)$ and we denote by $\chi_0(t)$ the symplectomorphism that transforms the perturbed Hamiltonian $H(\theta,I;t)$ into Birkhoff normal form, as constructed in Theorem \ref{main1}.
For the purpose of stating the quantum Birkhoff normal form for $\mathcal{P}_h(t)$, the Maslov class of the KAM tori $\{\Lambda_\omega:\omega\in\Omega_\kappa\}$ (as defined in Section 3.4 of \cite{duistbook}) can be identified with elements of $\vartheta\in H^1(\T^n;\Z)$ via the family of symplectomorphisms $\chi_0(t)\circ\chi_1:\T^n\times D\rightarrow T^*M$.
Following \cite{popov2} and \cite{colin}, we can then associate a smooth line bundle $\L$ over $\T^n$ with the class $\vartheta$, such that smooth sections $f\in\mathcal{C}^\infty(\T^n,\L)$ can be canonically identified with smooth functions $\tilde{f}\in\mathcal{C}^\infty(\R^n,\C)$ satisfying the quasiperiodicity condition
\begin{equation}
\label{maslovcanonical}
\tilde{f}(x+2\pi p)=\exp\left(\frac{i\pi}{2}\langle\vartheta,p\rangle\right)\tilde{f}(x)
\end{equation}
for all $p\in \Z^n$.

The quantum Birkhoff normal form in \cite{popovquasis} is far sharper than is necessary for the purposes of this paper, with remainders of order $O(e^{-ch^{-1/\nu}})$.  We require only the following truncated version, with error terms of order $O(h^{\gamma+1})$ for some fixed $\gamma>0$.

\begin{thm}
	\label{main2}
	Suppose $\mathcal{P}_h(t)$ is as in \eqref{operatorP}. Then for each fixed $t$,  there exists a uniformly bounded family of semiclassical Fourier integral operators 
	\begin{equation}
	U_h(t):L^2(\T^n;\mathbb{L})\rightarrow L^2(M)\quad (0<h<h_0)
	\end{equation}
	that are associated with the canonical relation graph of the Birkhoff normal form transformation $\chi(t)$ such that we have
	\begin{enumerate}
		\item $U_h(t)^*U_h(t)-\textrm{Id}$ is a pseudodifferential operator with symbol in the Gevrey class $S_\ell(\T^n\times D)$ which restricts to an element of $h^{\gamma+1}S_\ell(\T^n\times Y)$ for some subdomain $Y$ of $D$ that contains $E_\kappa(t)$.
		\item $\mathcal{P}_h(t)\circ U_h(t)-U_h(t)\circ \mathcal{P}^0_h(t)=\mathcal{R}_h(t)\in h^{\gamma+1}S_\ell$, where the operator $\mathcal{P}^0_h(t)$ has symbol 
		\begin{equation}
		\label{fioconjflatness}
		p^0(\theta,I;t,h)=K^0(I;t,h)+R^0(\theta,I;t,h)=\sum_{j\leq \gamma}K_j(I;t)h^j+\sum_{j\leq \gamma}R_j(\theta,I;t)h^j
		\end{equation}
		with both $K^0$ and $R^0$ in the symbol class $S_\ell(\T^n\times D)$ from Definition \ref{selldef} where $\eta>0$ is a constant, $K_0(I;t),R_0(\theta,I;t)$ are the components of the Birkhoff normal form of the Hamiltonian $P_0\circ \chi_1$ as constructed in Theorem \ref{main1}, and
		\begin{equation}
		\label{qbnfflatness}
		\partial_I^\alpha R_j(\theta,I;t)=0 
		\end{equation}
		for $(\theta,I;t)\in \T^n\times E_\kappa(t)\times (-1,1)$.
		Moreover, the symbols $K_j,R_j$ in \eqref{fioconjflatness} are smooth in the parameter $t$.
	\end{enumerate}
\end{thm}

Our statement of Theorem \ref{main2} differs from \cite{popovquasis} Theorem 2.1 only in the presence of the parameter $t$, the smoothness of the symbols $K_j,R_j$ in $t$, and the truncation to fixed finite order $O(h^{\gamma+1})$. We sketch the details of the proof of Theorem \ref{main2} in this section, following the argument of Popov \cite{popovquasis}.

The construction of $U_h(t)$ can be broken into multiple steps. We begin by constructing a family of semiclassical Fourier integral operators $T_h(t)$ that conjugate $P_h(t)$ to a family of semiclassical pseudodifferential operators $P_h^1(t):\mathcal{C}^\infty(\T^n;\mathbb{L})$ with principal symbol equal to $K_0(I;t)+R_0(\theta,I;t)$, the Birkhoff normal form of $H$, and with  vanishing subprincipal symbol. The conjugating semiclassical Fourier integral operators arise by quantising the $G^\rho$ symplectomorphisms \begin{equation}
\chi_1:\mathbb{T}^n\times D \rightarrow T^*M
\end{equation}
and 
\begin{equation}
\chi_0:\mathbb{T}^n\times D \rightarrow \mathbb{T}^n\times D
\end{equation} 
that transform the unperturbed Hamiltonian $P(x,\xi;0)$ to action-angle variables and transform the perturbed Hamiltonian to Birkhoff normal form respectively, and composing these two operators. Full details for this construction can be found in \cite{popov2} Section 2.

From the regularity of the symplectomorphisms, it follows that there exists a semiclassical expansion for $P_h^1(t)$ with symbols smooth in $t$.

The symbol of the operator $P_h^1(t)$ satisfies the property $\eqref{fioconjflatness}$ to $O(h^2)$, and to improve this, we replace the conjugating Fourier integral operator $T_h$ with $T_hA_h$ for a suitable elliptic pseudodifferential operator $A_h$ whose symbol is determined iteratively on the family of Cantor-like sets $\{(\theta,I;t)\in \T^n\times \R^n\times (-1,1):I\in E_\kappa(t)\}$ by solving equations of the form
\begin{equation}
\label{firsthomeq}
\langle \nabla K_0,\partial_\theta\rangle f(\theta,I;t)=g(\theta,I;t)
\end{equation}
referred to in the literature as homological equations. In this manner the ``flatness condition" of \eqref{qbnfflatness} is obtained for $j>0$, where the $j=0$ statement is established by Theorem \ref{main1}. We outline this procedure in Section \ref{ellconj}.

The key fact is that the homological equation can be solved smoothly in the parameter $t$, which is the content of Theorem \ref{homeqthm}. One can then apply Theorem \ref{homeqthm} as in \cite{popovquasis} Section 2.3 to complete the construction of the quantum Birkhoff normal form, with the additional consequence of smoothness of symbols $K_j,R_j$.

\subsection{Construction of the quantum Birkhoff normal form}
\label{ellconj}
After conjugating $P_h(t)$ by semiclassical Fourier integral operators as described in the previous section, we obtain a family of self-adjoint semiclassical operators $P_h^1(t)$ with symbol $\tilde{p}\in S_{\tilde{\ell}}(\mathbb{T}^n\times D)$ satisfying the flatness condition \eqref{qbnfflatness} to order $h^2$, where $\tilde{\ell}=(\rho,\rho',\rho+\rho'-1)$. That is to say, the formal summation of $\tilde{p}$
\begin{equation}
\sum_{j=0}^\infty \tilde{p}_j(\theta,I;t)h^j
\end{equation}
satisfies
\begin{equation}
\tilde{p}_0(\theta,I;t)=K_0(I;t)+R_0(\theta,I;t)
\end{equation}
and
\begin{equation}
\tilde{p}_1(\theta,I;t)=0.
\end{equation}

The next step of the proof of Theorem \ref{main2} is the improvement of the order of the flatness condition by composition with a suitable elliptic semiclassical pseudodifferential operator $A_h(t)=Id+O(h)$ with symbol
\begin{equation}
a(\theta,I;t)=\sum_{j=1}^\infty a_j(\theta,I;t)h^j.
\end{equation}
To motivate the method, we suppose that a quantum Birkhoff normal form $P_h^0$ exists in the sense of Theorem \ref{main2}. Our current operator $\tilde{P}_h$ is equal to $P_h^0$ up to order $h^2$ by construction. Hence, we have
\begin{eqnarray}
T_h(t)A_h(t)\tilde{P}_h(t)&=& T_h(t)\tilde{P}_h(t)A_h(t)+ T_h(t)[A_h(t),\tilde{P}_h(t)]\\
&=& P^1_h(t)T_h(t)A_h(t)+h^2T(t)B(t)A(t)+T_h(t)[A_h(t),\tilde{P}_h(t)].
\end{eqnarray}
for some semiclassical pseudodifferential operator $B_h(t)$ in the symbol class $S_{\tilde{\ell}}(\mathbb{T}^n\times D)$.
From composition formulae, the symbol of the commutator is equal to
\begin{equation}
-(\partial_\theta^\alpha a_1 \partial_I^\alpha \tilde{p}_0)h^2=-\mathcal{L}_{\omega{I;t}}a_1
\end{equation}
where $\mathcal{L}_\omega=\langle \omega,\partial_\theta\rangle a_1(\theta,I;t)$. Thus to improve the order of the flatness condition, it suffices to choose $a_1$ solving the homological equation
\begin{equation}
\label{homeqex}
\mathcal{L}_{\omega(I;t)}a_1=b_0
\end{equation}
where $b_0$ denotes the principal symbol of $B_h(t)$. Indeed, if \eqref{homeqex} is solvable, then we have
\begin{equation}
T_h(t)A_h(t)P_h(t)=P_h^0(t)T_h(t)A_h(t)+O(h^3).
\end{equation}

Extending this idea, it is shown by Popov \cite{popov2} that we can choose higher order terms of the symbol $a$ in an iterative fashion by the solution of such a homological equation for each power of $h$ that we gain. The consequence is the following result.
\begin{prop}
	\label{symbolconst}
	There exists $a,K^0,r\in S_\ell(\mathbb{T}^n\times D)$ where $\ell=(\rho,\mu,\nu)$ such that
	\begin{equation}
	\label{aexpansion}
	a(\theta,I;t,h)\sim \sum_{j=0}^\infty a_j(\theta,I;t)h^j
	\end{equation}
	\begin{equation}
	\label{Kexpansion}
	K^0(I;t,h)\sim \sum_{j=0}^\infty K_j(I;t)h^j
	\end{equation}
	and
	\begin{equation}
	r(\theta,I;t,h)\sim \sum_{j=0}^\infty r_j(\theta,I;t)h^j
	\end{equation}
	where $a_0=1,r_0=R_0,K_1=0,$ and 
	\begin{equation}
	\tilde{p}\circ a-a\circ K^0\sim r.
	\end{equation}
	where each $r_j(\theta,I;t)$ is flat in $I$ on $\T^n\times E_\kappa(t).$
\end{prop}

The symbol $K^0$ in the statement of theorem corresponds to the sought symbol $K^0$ in Theorem \ref{main2}, while the symbol $R^0$ is then constructed by solving $a\circ R^0=r$, which is possible by ellipticity.

The completion of the proof of Theorem \ref{main2} after establishing Proposition \ref{symbolconst} is contained in \cite{popov2} Section 3. For our additional requirement of smoothness in $t$ in Theorem \ref{main2}, it thus suffices to verify that the homological equation can be solved smoothly in the parameter $t$.
In particular, we require the following.

\begin{thm}
	\label{homeqthm}
	Suppose $f(\cdot,\cdot;t)\in G^{\rho,\mu}(\mathbb{T}^n\times D)$ satisfies the estimate
	\begin{equation}
	\label{homeqhyp}
	|\partial_\theta^\alpha\partial_I^\beta f(\theta,I;t)|\leq d_0 C^{|\alpha|+\mu|\beta|}\Gamma(\rho|\alpha|+\mu|\beta|+q)
	\end{equation}
	uniformly in the smooth parameter $t\in (-1,1)$ for some $q>0$ and some $C\geq 1$ and that for each $I\in D$, we have
	\begin{equation}
	\label{homeqavg}
	\int_{\T^n}f(\theta,I;t)\, d\theta=0.
	\end{equation}
	
	Then for any smooth family $\omega(\cdot;t) \in G^{\rho'}_{L_0}(D,\Omega)$ there is a solution $u(\cdot,\cdot;t)\in G^{\rho,\mu}(\T^n\times D)$ to the equation
	\begin{eqnarray}
	\label{homeqn}
	\mathcal{L}_\omega u(\theta,I;t)&=&f(\theta,I;t)\quad (\theta,I)\in \T^n\times E_\kappa(t)\\
	u(0,I;t)&=&0\quad I\in D
	\end{eqnarray}
	where $\mathcal{L}_\omega=\langle \omega(I;t),\frac{\partial}{\partial \theta}\rangle $.
	Moreover, $u$ is smooth in the parameter $t$ and satisfies the estimate
	\begin{equation}
	\label{homeqestimate}
	|\partial_\theta^\alpha \partial_I^\beta u(\theta,I;t)|\leq Ad_0C^{n+\tau+|\alpha|+\mu|\beta|+1 }\Gamma(\rho|\alpha|+\mu|\beta|+\rho(n+\tau+1)+q)
	\end{equation}
	where $A$ depends only on $n,\rho,\tau$ and $\mu$.
\end{thm}

	This theorem statement differs from \cite{popovquasis} Proposition 2.3 only in the presence of the smooth parameter $t$, and indeed an identical proof based on taking the Fourier expansion
	\begin{equation}
	u(\theta,I;t)=\sum_{k\in \Z^n} e^{i\langle k,\theta \rangle} u_k(I;t)
	\end{equation}
	and solving for $u_k$ can be pursued. The rapid decay of Fourier coefficients established in \cite{popovquasis} implies that that the limit $u(\theta,I;t)$ is smooth in $t$ as required. The proof is then identical to that in \cite{popovquasis}, with the uniformity in \eqref{homeqestimate} following from the uniformity in \eqref{homeqhyp}.

\subsection{Quasimode construction}
\label{gevquasisec}
We now briefly outline how the construction of Gevrey class quasimodes for $\mathcal{P}_h(t)$ follow from the quantum Birkhoff normal form Theorem \ref{main2}.
These quasimodes microlocalise onto a family of nonresonant tori and moreover have quasi-eigenvalues are smooth in the parameter $t\in (-1,1)$.

\begin{defn}
	\label{gevmodes}
	An $O(h^{\gamma+1})$ family of $G^\rho$ quasimodes $\mathcal{Q}(t)$ for $P_h(t)$ is a family
	\begin{equation}
	\{(u_m(x;t,h),\lambda_m(t,h)):m\in\mathcal{M}_h(t)\} \subset \mathcal{C}^\infty(M\times \mathcal{D}_h(m))\times \mathcal{C}^\infty(\mathcal{D}_h(m))
	\end{equation}
	parametrised by $h\in (0,h_0]$ where
	\begin{itemize}
		\item $\mathcal{M}_h(t)\subset \Z^n$ is a $h$-dependent finite index set;
		\item $\mathcal{D}_h(m)=\{t\in (-1,1):m\in \mathcal{M}_h(t)\}$
		\item each $u(\cdot;t,h)$ is uniformly of class $G^\rho$;
		\item $\|P_h(t)u_m(\cdot;t,h)-\lambda_m(t;h)u_m(\cdot;t,h)\|_{L^2}=O(h^{\gamma+1}) \quad \forall m\in\mathcal{M}_h(t) $;
		\item $|\langle u_m(\cdot;t,h),u_l(\cdot;t,h)\rangle -\delta_{ml}|=O(h^{\gamma+1}) \quad \forall m,l\in\mathcal{M}_h(t) $.
	\end{itemize}
\end{defn}

\begin{thm}
	\label{gevquasisthm}
	Suppose now that $t\in (-1,1)$ is fixed and $S\subset E_\kappa(t)$ is a closed collection of nonresonant actions. For an arbitrary constant $L>1$, we define the index set
	\begin{equation}
	\label{indexsetref}
	\mathcal{M}_h:=\{m\in \Z^n: \textrm{dist}(S,h(m+\vartheta/4))< Lh\}
	\end{equation}
	where $\vartheta\in \Z^n$ is the Maslov class of any Lagrangian tori $\{\chi(\T^n\times \{I\})\}$ with $I\in S$. Note that this class is independent of choice of torus by the local constancy of the Maslov class.
	
	Then 
	\begin{equation}
	\{(u_m(x;t,h),\lambda_m(t;h)):m\in\mathcal{M}_h(t)\}:= (U_h(t)e_m,K^0(h(m+\vartheta/4);t,h)
	\end{equation}
	defines a $G^\rho$ family of quasimodes for $P_h(t)$ that has Gevrey microsupport on the family of tori 
	\begin{equation}
	\displaystyle
	\Lambda_S=\bigcup_{I\in S} \Lambda_{\omega(I;t)}=\bigcup_{I\in S}\chi(\T^n\times \{I\})\subset T^*M
	\end{equation}
	where $\{e_m\}_{m\in\Z^n}$ is the orthonormal basis of $L^2(\T^n;\mathbb{L})$ associated to the quasiperiodic functions
	\begin{equation}
	\tilde{e}_m(x):=\exp(i\langle m+\vartheta/4,x \rangle)
	\end{equation}
\end{thm}
\begin{proof}
	From the definition of the functions $e_m$, it follows that
	\begin{eqnarray}
	P_h^0(t)(e_m)(\theta)&=& \sigma(P_h^0(t))(\theta,h(m+\vartheta/4))e_m(\theta)\\
	&=& (K^0(h(m+\vartheta/4);t,h)+R^0(\theta,h(m+\vartheta/4);t,h))e_m(\theta)\\
	&=& (\lambda_m(t;h)+R^0(\theta,h(m+\vartheta/4))e_m(\theta).
	\end{eqnarray}
	From the definition \eqref{indexsetref} of the index set $\mathcal{M}_h(t)$ and from $\eqref{flatissmall}$, it thus follows that 
	\begin{equation}
	P_h(t)(U_h(t)e_m)=U_h(t)P_h^0(t)e_m=O(h^{\gamma+1})
	\end{equation}
	upon an application of Theorem \ref{main2}. 
	The almost-orthogonality of the $U_h(t)e_m$ then follows from the fact that $U_h(t)$ is almost unitary from Theorem \ref{main2}, and that the $e_m$ are exactly orthogonal by construction.
	This completes the proof.
\end{proof}
These quasimodes are as numerous as we could hope for, indeed the index set $\mathcal{M}_h(t)$ satisfies the local Weyl asymptotic
\begin{equation}
\label{indexsetasymp}
\lim_{h\rightarrow 0}(2\pi h)^n \# \mathcal{M}_h =m(\T^n\times S)= \mu(\Lambda_S)
\end{equation}
where $m$ denotes the $(2n)$-dimensional Lebesgue measure and $\mu$ denotes the symplectic measure $d\xi\,dx$. To see this, we can denote by $U$ the union of $n$-cubes centred at the lattice points in $\mathcal{M}_h$ with side length $h$. The containment
\begin{equation}
\label{latter}
S\subset U \subset \{I:\textrm{dist}(I,S)<\tilde{L}h\}
\end{equation}
for a constant $\tilde{L}$ then yields the claim by monotone convergence of measures, noting that since $S$ is closed we have
\begin{equation}
S=\overline{S}=\cap_{h>0} \{I:\textrm{dist}(I,S)<\tilde{L}h\}.
\end{equation}
In the special case of $S=\{I\}$, we have a family of $G^\rho$ quasimodes with microsupport on an individual torus $\chi(\T^n\times \{I\})$. 




\appendix

\section{Anisotropic Gevrey classes}
\counterwithin{thm}{section}
\counterwithin{equation}{section}
\label{gevreyappendix}

In this appendix, we define the Gevrey function spaces used throughout the paper and collect several of their properties from the appendix of \cite{popovkam}. 
\begin{defn}
	\label{gevrey}
	For $\rho\geq 1$ and $X\subset \mathbb{R}^n$ open, the \textsf{Gevrey class of order }$\rho$ is given by
	\begin{equation}
	\label{gevreyeqn}
	G^\rho_L(X):=\{f\in\mathcal{C}^\infty(X): \sup_\alpha\sup_{x\in X} |\partial_x^\alpha f(x)|L^{-|\alpha|}\alpha!^{-\rho}< \infty\}.
	\end{equation}
\end{defn}
If $f\in G^\rho_L(X)$, the supremum in \eqref{gevreyeqn} is denoted by $\|f\|_L$. We will frequently suppress the $L$ in our notation. Equipped with this norm, $G^\rho_L(X)$ is a Banach space.
Gevrey regularity is generally weaker the real analyticity (they coincide when $\rho=1$ as can be seen by using the Cauchy--Hadamard theorem to characterise analytic functions by the growth of their Taylor coefficients) and importantly, there exist bump functions in the Gevrey class for $\rho>1$. 

An important property of the Gevrey class that follows from Taylor's theorem is that if a Gevrey function has vanishing derivatives, then locally it is super-exponentially small.
\begin{prop}
	\label{flatissmall}
	Suppose $f\in G^\rho(X)$, and $\rho>1$.
	Then there exist positive constants $c,C,\eta$ and $r_0$ only dependent on the Gevrey constant $L$, the norm $\|f\|_L$, and the set $X$ such that
	\begin{equation}
	f(x_0+r)=\sum_{|\alpha|\leq \eta |r|^{1/(1-\rho)}}f_\alpha(x_0)r^\alpha +R(x_0,r)
	\end{equation}
	where $f_\alpha=(\partial^\alpha f)/\alpha!$ and
	\begin{equation}
	|\partial_x^\beta R(x_0,r)|\leq C^{1+|\beta|}\beta!^\rho e^{-c|r|^{-1/(\rho-1)}}\quad \forall 0<|r|\leq \min(r_0,d(x_0,\R^n\setminus X)).
	\end{equation}
\end{prop}
We also need to consider anisotropic Gevrey classes, which are classes of Gevrey functions with differing regularity in individual variables.
\begin{defn}
	\label{gevrey2}
	Suppose $X$ and $Y$ are open subsets of Euclidean spaces. Suppose that $\rho_1,\rho_2\geq 1$ and $L_1,L_2>0$. Then
	\begin{equation}
	\label{anisotropicgevreyeqn}
	G^{\rho_1,\rho_2}_{L_1,L_2}(X\times Y)=\{f\in\mathcal{C}^\infty(X\times Y):\sup_{(x,y)\in X\times Y} |\partial_x^\alpha\partial_y^\beta f|  L_1^{-|\alpha|} L_2^{-|\beta|}\alpha!^{-\rho_1} \beta!^{-\rho_2}<\infty \}.
	\end{equation}
\end{defn}

If $f\in G^{\rho_1,\rho_2}_{L_1,L_2}$, then we denote the supremum in \eqref{anisotropicgevreyeqn} by $\|f\|_{L_1,L_2}$. Equipped with this norm, $G^{\rho_1,\rho_2}_{L_1,L_2}$ is a Banach space. This definition extends in the natural way to $k\geq 3$ variables. Furthermore, some of these variables might lie in complex domains.

In anisotropic Gevrey classes, one has the following implicit function theorem due to Komatsu.
\begin{prop}
	\label{komatsuprop}
	Suppose that $F\in G^{\rho,\rho'}_{L_1,L_2}(X\times \Omega^0,\R^n)$ where $X\subset \R^n$, $\Omega^0\subset \R^m$ and  $L_1\|F(x,\omega)-x\|_{L_1,L_2}\leq 1/2$. Then there exists a local solution $x=g(y,\omega)$ to the implicit equation
	\begin{equation}
	F(x,\omega)=y
	\end{equation} 
	defined in a domain $Y\times \Omega$.
	Moreover, there exist constants $A,C$ dependent only on $\rho,\rho',n,m$ such that $g\in G^{\rho,\rho'}_{CL_1,CL_2}(Y\times \Omega,X)$ with $\|g\|_{CL_1,CL_2}\leq A\|F\|_{L_1,L_2}$.
\end{prop}
A consequence of this theorem is established by Popov in \cite{popovkam}.
\begin{cor}
	\label{komatsucor}
	Suppose $F\in G^{\rho,\rho'}_{L_1,L_2}(\T^n\times \Omega,\T^n)$ where $\Omega^0\subset \R^m$ and  $L_1\|F(\theta,\omega)-\theta\|_{L_1,L_2}\leq 1/2$. Then there exists a local solution $x=g(y,\omega)$ to the implicit equation
	\begin{equation}
	F(x,\omega)=y
	\end{equation} 
	defined on $\T^n\times \Omega$.
	Moreover, there exist positive constants $A,C$ dependent only on $\rho,\rho',n,m$ such that $g\in G^{\rho,\rho'}_{CL_1,CL_2}(\T^n\times \Omega)$ with $\|g\|_{CL_1,CL_2}\leq A\|F\|_{L_1,L_2}$.
\end{cor}
Finally, we have two results on the composition of functions of Gevrey regularity, which can also be found in \cite{popovkam}.
\begin{prop}
	\label{gevcomp1}
	Let $X\subset \R^n$, $Y\subset\R^m$, and $\Omega\subset \R^k$ be open sets.
	Suppose $g\in G^{\rho'}_{L_1}(\Omega,Y)$ with $\|g\|_{L_1}=A_1$ and $f\in G^{\rho,\rho'}_{B,L_2}(X\times Y)$ with $\|f\|_{B,L_2}=A_2$.
	Then the composition $F(x,\omega):=f(x,g(\omega))$ is in $G^{\rho,\rho'}_{B,L}(X\times \Omega)$, where $$L=2^{l+\rho'}l^{\rho'}L_1\max(1,A_1L_2)$$ with $l=\max(k,m,n)$.
	Moreover we have the Gevrey norm estimate $$\|F\|_{B,L}\leq A_2$$.
\end{prop}
\begin{prop}
	\label{gevcomp2}
	Let $X\subset \R^n$, $Y\in\R^m$, and $\Omega\subset \R^k$ be open sets.
	Suppose $g\in G^{\rho,\rho'}_{B_1,L_1}(X\times \Omega,Y)$ with $\|g\|_{B_1,L_1}=A_1$ and $f\in G^{\rho,\rho'}_{B_2,L_2}(Y\times \Omega)$.
	Then the composition $F(x,\omega):=f(g(x,\omega),\omega)$ is in $G^{\rho,\rho'}_{B,L}(X\times \Omega)$, where $$B=4^l(4l)^\rho B_1 \max(1+A_1B_2)$$ and $$L=L_2+4^l(4l)^\rho L_1 \max(1,A_1B_2)$$ with $l=\max(k,m,n)$.
	Moreover we have the Gevrey norm estimate $$\|F\|_{B,L}\leq A_2.$$
\end{prop}

\section{Gevrey class symbols}
\counterwithin{thm}{section}
\counterwithin{equation}{section}
\label{gevsymbsec}

In this appendix, we introduce the class of Gevrey symbols used throughout this paper. 
We suppose $D$ is a bounded domain in $\mathbb{R}^n$, and take $X=\mathbb{T}^n$ or a bounded domain in $\mathbb{R}^m$.
We fix the parameters $\sigma,\mu>1$ and $\varrho\geq \sigma+\mu-1$, and denote the triple $(\sigma,\mu,\varrho)$ by $\ell$.

\begin{defn}
	\label{formalsymboldefn}
	A formal Gevrey symbol on $X\times D$ is a formal sum
	\begin{equation}
	\label{formalsymbol}
	\sum_{j=0}^\infty p_j(\theta,I)h^j
	\end{equation}
	where the $p_j\in\mathcal{C}_0^\infty(X\times D)$ are all supported in a fixed compact set and there exists a $C>0$ such that
	\begin{equation}
	\sup_{X\times D} |\partial_\theta^\beta \partial_I^\alpha p_j(\theta,I)|\leq C^{j+|\alpha|+|\beta|+1}\beta!^\sigma\alpha!^\mu j!^\varrho.
	\end{equation}
\end{defn}

\begin{defn}
	A realisation of the formal symbol \eqref{formalsymbol} is a function $p(\theta,I;h)\in\mathcal{C}_0^\infty(X\times D)$ for $0<h\leq h_0$ with
	\begin{equation}
	\sup_{X\times D \times (0,h_0]} \left|\partial_\theta^\beta \partial_I^\alpha \left(p(\theta,I;h)-\sum_{j=0}^N  p_j(\theta,I)h^j\right)\right|\leq h^{N+1}C_1^{N+|\alpha|+|\beta|+2}\beta!^\sigma\alpha!^\mu (N+1)!^\varrho.
	\end{equation}
\end{defn}

\begin{lem}
	Given a formal symbol \eqref{formalsymbol}, one choice of realisation is
	\begin{equation}
	p(\theta,I;h):= \sum_{j\leq \epsilon h^{-1/\varrho}} p_j(\theta,I)h^j
	\end{equation}
	where $\epsilon$ depends only on $n$ and $C_1$.
\end{lem}

\begin{defn}
	We define the residual class of symbols $S_\ell^{-\infty}$ as the collection of realisations of the zero formal symbol.
\end{defn}

\begin{defn}
	\label{selldef}
	We write $f\sim g$ if $f-g\in S_\ell^{-\infty}$. It then follows that any two realisations of the same formal symbol are $\sim$-equivalent.
	We denote the set of equivalence classes by $S_\ell(X\times D)$.
\end{defn}

We now discuss the class of pseudodifferential operators corresponding to these symbols.
\begin{defn}
	\label{gevpseudo}
	To each symbol $p\in S_\ell(X\times D)$, we associate a semiclassical pseudodifferential operator defined by
	\begin{equation}
	(2\pi h)^{-n}\int_{X\times \mathbb{R}^n}e^{i(x-y)\cdot \xi/h}p(x,\xi;h)u(y)\, d\xi\, dy.
	\end{equation}
	for $u\in \mathcal{C}_0^\infty(X)$.
\end{defn}

The above construction is well defined modulo $\exp(-ch^{-1/\varrho})$, as for any $p\in S_\ell^{-\infty}(X\times D)$ we have
\begin{equation}
\|P_hu\|=O_{L^2}(\exp(-ch^{-1/\varrho}))
\end{equation}
for some constant $c>0$.
\begin{remark}
	The exponential decay of residual symbols is a key gain that comes from working in a Gevrey symbol class.
\end{remark}

The operations of symbol composition and conjugation then correspond to composing operators and taking adjoints respectively.
Moreover, if $p\in S_{(\sigma,\sigma,2\sigma-1)}$, then $G^\sigma$-smooth changes of variable preserve the symbol class of $p$.
This coordinate invariance allows us to extend the Gevrey pseudodifferential calculus to compact Gevrey manifolds.

\section{Estimates for analytic functions}
\counterwithin{thm}{section}
\counterwithin{equation}{section}
\label{analytic}

In this appendix we prove several elementary but important estimates for analytic functions.
\begin{prop}
	\label{cauchyappendix}
	Suppose $\tilde{\Omega}_j \subset \C$ are open sets and $\Omega_j\subset \tilde{\Omega}_j$ are such that $\textrm{dist}(\Omega_j,\C\setminus \tilde{\Omega_j})<r_j$. 
	
	Define
	\begin{equation}
	\Omega=\prod_{j=1}^n \Omega_j
	\end{equation}
	and
	\begin{equation}
	\tilde{\Omega}=\prod_{j=1}^n \tilde{\Omega}_j.
	\end{equation}
	
	If the analytic function $f:\tilde{\Omega}^n\rightarrow \mathbb{C}$ satisfies 
	\begin{equation}
	\|f\|_{\Omega}=A < \infty
	\end{equation}
	then we have
	\begin{equation}
	\|\partial^\alpha_zf\|_\Omega \leq Ar^{-\alpha} \alpha!
	\end{equation}
	for each multi-index $\alpha$.
\end{prop}
\begin{proof}
	From the Cauchy integral formula, we have
	\begin{equation}
	f(z)=\frac{1}{(2\pi i)^n}\oint_{\partial B(z_1,r_1)}\oint_{\partial B(z_2,r_2)}\ldots \oint_{\partial B(z_n,r_n)} \frac{f(w)}{w-z}\, dw_1\, dw_2 \ldots \, dw_n.
	\end{equation}
	which yields
	\begin{equation}
	\partial^\alpha_z f(z)=\frac{\alpha!}{(2\pi i)^n}\oint_{\partial B(z_1,r_1)}\oint_{\partial B(z_2,r_2)}\ldots \oint_{\partial B(z_n,r_n)} \frac{f(w)}{(w-z)^{\alpha+1}}\, dw_1\, dw_2 \ldots \, dw_n.
	\end{equation}
	upon repeated differentiation, where $1$ denotes the multi-index $(1,1,\ldots,1)$.
	Hence 
	\begin{equation}
	\|\partial^\alpha_zf\|_\Omega \leq Ar^{-\alpha} \alpha!
	\end{equation}
	as required.
\end{proof}
We also have an implicit function theorem for real analytic functions. Defining
\begin{equation}
O_h=\{\omega\in \C^n:\textrm{dist}(\omega,\Omega)<h\}
\end{equation}
where distances in $\C^n$ are taken with the sup-norm, we have the following.
\begin{prop}
	\label{popovlemma}
	Suppose $f:O_h\times (-1,1)\rightarrow \C^n$ is real analytic, and we have the estimate
	\begin{equation}
	|f|_{h}<\infty,
	\end{equation}
	then for any $0<v<1/6$ such that
	\begin{equation}
	|f-id|_h\leq vh
	\end{equation}
	the function has a real analytic inverse $g:O_{(1/2-3v)h}\times (-1,1)\rightarrow O_{(1-4v)h}$ that satisfies the estimate
	\begin{equation}
	\label{appendixnormderiv}
	\max(|g-id|_{(1/2-3v)h},3vh|D\phi-Id|_{(1/2-3v)h})\leq |f-id|_h
	\end{equation}
	uniformly in $t\in (-1,1).$
	The matrix norm in \eqref{appendixnormderiv} is the norm induced by equipping $\C^n$ with the sup-norm.
\end{prop}
Proposition \ref{popovlemma} can be proven in the same way as in Lemma 3.4 of \cite{popovkam}. The only difference is that we need to work on domains of the form $O_{\lambda h}\times B^\mathbb{C}_1$, and invert maps of the form 
\begin{equation}
\tilde{f}(\omega,t):=(f(\omega,t),t)
\end{equation}
for given $f$ satisfying the assumptions of the proposition uniformly in $t$.

\section{Whitney extension theorem}
\label{whitneysec}
\counterwithin{thm}{section}
\counterwithin{equation}{section}

In this appendix, we prove a version of the Whitney extension theorem for anisotropic Gevrey classes. The proof is adapted from the work of Bruna \cite{bruna} in the case without in the non-anisotropic case.

\begin{defn}
	\begin{equation}
	\label{anisspace}
	\Cm=\{f\in\mathcal{C}^\infty(X\times Y,\R): \sup_{(x,y)\in X\times Y}\sup_{\alpha,\beta}\left(\frac{|(\partial_x^\alpha \partial_y^\beta f)(x,y)|}{L_1^{|\alpha|}L_2^{|\beta|}M_{|\alpha|}\tilde{M}_{|\beta|}}\right) < \infty \textrm{ for some }L_j>0\}
	\end{equation}
\end{defn}
where $X,Y$ are open sets in Euclidean spaces of possibly differing dimension, $\alpha,\beta$ are multi-indices of the appropriate dimension, and $M$ and $\tilde{M}$ are positive sequences satisfying
\begin{enumerate}
	\item $M_0=1$
	\item $M_k^2\leq M_{k-1}M_{k+1}$
	\item $M_k\leq A^k M_j M_{k-j}$
	\item $M_{k+1}^k \leq A^k M_k^{k+1}$
	\item $M_{k+1}/(kM_k)$ is increasing
	\item $\sum_{k\geq 0}M_k/M_{k+1}\leq ApM_p/M_{p+1}$ for $p>0$
\end{enumerate}
where $A>0$ is a positive constant.

In the Gevrey case of interest to us, $M_k=k!^{\rho_1},\tilde{M}_k=k!^{\rho_2}$.
For fixed $L_j>0$, the supremum in \eqref{anisspace} defines a norm which equips a subspace of $\Cm$ with a Banach space structure.
The space $\Cm$ is then the inductive limit of these spaces as $L=L_1=L_2\rightarrow\infty$, which identifies it a Silva space.

For $f\in \Cm$, and $z=(z_1,z_2)\in X\times Y, x\in X$ we define
\begin{defn}
	\label{mixedtaylor}
	\begin{equation}
	(T^{m}_xf)(z):=\sum_{|\alpha|\leq m}\frac{(\partial_x^{\alpha} f)(x,z_2)}{\alpha!}(z_1-x)^{\alpha}
	\end{equation}
\end{defn}
\begin{defn}
	\label{mixedremainder}
	\begin{equation}
	(R^{m}_xf)(z):=f(z)-(T^{m}_xf)(z).
	\end{equation}
\end{defn}
To slightly generalise this notation, for a jet $f^{\alpha,\beta}$ of continuous functions, we write
\begin{defn}
	\label{jetremainder}
	\begin{equation}
	(R^{m}_{x}f)_{\alpha,\beta}(z):=f^{\alpha,\beta}(z)-(T^{m-|\alpha|}_{x}f^{\alpha,\beta})(z)
	\end{equation}
\end{defn}

We can now pose the central question:

\medskip 
\fbox{\parbox{\textwidth}{
		Given a compact set $K\subset X$, under what conditions is it true that an arbitrary continuous jet $(f^{\alpha,\beta}):K\times Y\rightarrow \R$ is the jet of a function $\tilde{f}\in\Cm$?
}}
\medskip

We assume without loss of generality here that the set $X$ is a full Euclidean space $\R^d$, rather than just an open subset thereof. This question is the anisotropic non quasi-analytic analogue of Whitney's extension theorem from classical analysis, which deals with the $\mathcal{C}^\infty$ case.

We begin by finding \emph{necessary} conditions for the existence of such an extension, before proving that these conditions are indeed sufficient.

\begin{prop}
	\label{anistaylor}
	Suppose $f\in \Cm$ with Gevrey constants $L_1,L_2$. Then there exists a constant $A$ dependent only on the dimensions of $X,Y$ and on $M,\tilde{M}$ such that the jet $f^{\alpha,\beta}=\partial_z^{(\alpha,\beta)}f$ satisfy
	\begin{equation}
	\label{gevcond}
	|f^{\alpha,\beta}|\leq AL_1^{|\alpha|}L_2^{|\beta|}M_{|\alpha|}\tilde{M}_{|\beta|}
	\end{equation}
	and
	\begin{equation}
	\label{gevcond2}
	|(R^n_xf)_{k,l}(z)|\leq A\tilde{L}_1^{n+1}M_{n+1}L_2^{|l|}\tilde{M}_{|l|}\cdot \frac{|z_1-x|^{n+1}}{(n+1)!}
	\end{equation}
	for all non-negative integers $m,n$ and all multi-indices $|k|\leq m,|l|\leq n$, where $\tilde{L}_1=CL_1$ with the $C$ dependent only on the dimension of $X$.
\end{prop}
\begin{proof}
	The first estimate \eqref{gevcond} follows immediately from the definition of $\Cm$.
	We prove the second claim \eqref{gevcond2} by making use of the estimate \eqref{gevcond} on the jet $f^{\alpha,\beta}=\partial_x^\alpha \partial_y^\beta f$ and Taylor expansion.
	\begin{eqnarray}
	R_x^nf(z)&=& \sum_{|\alpha|=n+1}\frac{n+1}{\alpha!} (z_1-x)^{\alpha}\int_0^1 (1-t)^nf^{\alpha,0}(x+t(z_1-x),z_2)\, dt\\
	&\leq & \left(\sup_{|\alpha|=n+1}\sup_{z\in X\times Y} |f^{\alpha,0}(z)|\right)\cdot \sum_{|\alpha|=n+1}\left|\frac{(z_1-x)^\alpha}{\alpha!}\right|\\
	&\leq & \left(\sup_{|\alpha|=n+1}\sup_{z\in X\times Y} |f^{\alpha,0}(z)|\right)\cdot \frac{C^{n+1}|z_1-x|^{n+1}}{(n+1)!}\\
	\end{eqnarray}
	Hence
	\begin{equation}
	|(R^n_xf)_{k,l}(z)|=|(R_x^{n-|k|}f)(z)|\leq A\tilde{L}_1^{n+1}M_{n+1}L_2^{|l|}\tilde{M}_{|l|}\cdot \frac{|z_1-x|^{n+1}}{(n+1)!}
	\end{equation}
	as required.
\end{proof}
Subsequently, for simplicity of notation, we omit the tilde in $\tilde{L}_1$ with the understanding that we are allowed to absorb constants that are dependent only on the dimensions of $X,Y$ and on the sequences $M,\tilde{M}$.
\begin{thm}
	\label{whitneymainthm}
	Suppose $(f^{\alpha,\beta}):K\times Y\rightarrow \R$ is a jet of continuous functions smooth in $y$ that satisfies
	\begin{equation}
	\partial_y^\gamma(f^{\alpha,\beta})=f^{\alpha,\beta+\gamma}
	\end{equation} 
	as well as the conditions \eqref{gevcond} and \eqref{gevcond2} on $K\times Y$. Then there exists a function $f\in \Cm$ such that $\partial^{\alpha,\beta}_x f=f^{\alpha,\beta}$ on $K\times Y$.
	
	Moreover, there exist constants $C_0,C_1$ dependent only on the dimensions of $X$ and $Y$ and the weight sequences $(M_k),\tilde{M}_k$ such that
	\begin{equation}
	\|f\|_{C_1L_1,L_2}\leq C_0A.
	\end{equation}
\end{thm}
Before proving Theorem \ref{whitneymainthm}, we need to collect some lemmas, the proofs of which can be found in \cite{bruna}.
\begin{prop}
	\label{cubesprop}
	Suppose $K\subset \mathbb{R}^d$ is compact. Then there exists a collection of closed cubes $\{Q_j\}_{j\in\mathbb{N}}$ with sides parallel to the axes such that
	\begin{enumerate}
		\item $\R^d\setminus K=\cup_j Q_j$;
		\item $\textrm{int}(Q_j)$ are disjoint;
		\item $\delta_j:=\textrm{diam}(Q_j)\leq d_j:= d(Q_j,K)\leq 4\delta_j$;
		\item For $0<\lambda < 1/4$, $d(z,K)\sim\delta_j$ for $z\in Q_j^*:= (1+\lambda)Q_j$;
		\item Each $Q_i^*$ intersects at most $D=(12)^{2d}$ cubes $Q_j^*$;
		\item $\delta_i\sim\delta_j$ if $Q_i^*\cap Q_j^*\neq \emptyset$.
	\end{enumerate}
\end{prop}
\begin{prop}
	\label{pou}
	For each $\eta >0$, there exists a family of functions $\phi_i\in \mathcal{C}^\infty_M(\R^d)$ such that
	\begin{enumerate}
		\item $0\leq \phi_i$;
		\item $\textrm{supp}(\phi_i)\subset Q_i^*$;
		\item $\sum_i \phi_i(z)=1$ for $z\in \R^d$;
		\item $|\partial^\alpha \phi_i(z)|\leq Ah(B\eta d(z,K))\eta^{|\alpha|} M_{|\alpha|}$ for $z\in Q_i^*$.
	\end{enumerate}
	where $A,B>0$ are constants and 
	\begin{equation}
	h(t):=\sup_k\frac{k!}{t^kM_k}.
	\end{equation}
\end{prop}
\begin{prop}
	\label{openmapping}
	Suppose $T\in \mathcal{L}(E,F)$ is a continuous linear surjection between Silva spaces. Then for any bounded set $B\subset F$, there exists a bounded set $C\subset E$ with $T(C)=B$.
\end{prop}
We also require an anisotropic version of Carleman's theorem, which is the special case of \ref{whitneymainthm} with $K=\{0\}$, and Gevrey analogue of Borel's theorem from classical analysis.
\begin{prop}
	\label{carlemanvariant}
	Let $(g_\alpha)_{\alpha\in \N^{d}}$ be a multisequence of functions in $\mathcal{C}^\infty_{\tilde{M}}(Y)$ such that
	\begin{equation}
	|\partial_y^l g_\alpha(y)|\leq KL_1^{|\alpha|}L_2^{|l|}M_{|\alpha|}\tilde{M}_{|l|}.
	\end{equation}
	for some constant $K>0$.
	
	Then there exists a function $f\in \Cm$ such that $g_\alpha(y)=\partial_x^{\alpha}f(0,y)$ for all $y\in Y$. Moreover, $\|f\|_{CL_1,L_2}\leq AK$ for some constants $A,C>0$ independent of $f,L_1,$ and $L_2$.
\end{prop}
\begin{proof}
	We adapt the solution of \cite{pet} of the classical Carleman problem to this setting. Key is that the assumptions on $M$ imply that the hypotheses of \cite{pet} are satisfied. 
	Hence as in the proof of \cite{pet} Theorem 2.1 (ai), we can construct compactly supported $\chi_p(x)\in \mathcal{C}^\infty_{M_p}(\R)$ for each non-negative integer $p$ such that
	\begin{equation}
	\chi^{(k)}_p(0)=\delta(k,p)
	\end{equation}
	and
	\begin{equation}
	\|\chi_p\|_{L(2+A^{-1})}\leq \frac{1}{M_p}\cdot \left(\frac{Ae}{L}\right)^p
	\end{equation}
	for some dimensional constant $A$ and any $L>0$.
	Hence we can define
	\begin{equation}
	\chi_\alpha(x):=\prod_{j=1}^d \chi_{\alpha_j}(x_j)
	\end{equation}
	for $\alpha\in\N^d$ which satisfies
	\begin{equation}
	\chi_\alpha^{(\beta)}(0)=\delta(\beta,\alpha).
	\end{equation}
	Moreover, we have the estimate
	\begin{eqnarray}
	|\chi_\alpha^{(\beta)}|&= & \prod_{j=1}^d |\chi_{\alpha_j}^{\beta_j}|\\
	&\leq & \prod_{j=1}^d \frac{1}{M_{\alpha_j}}\left(\frac{Ae}{L}\right)^{\alpha_j}(L(2+A^{-1}))^{\beta_j}M_{\beta_j}\\
	&\leq & \left(\frac{Aec(d,M)}{L}\right)^{|\alpha|}\cdot M_{|\alpha|}^{-1}(L(2+A^{-1}))^{|\beta|}M_{|\beta|}.
	\end{eqnarray}
	
	By taking $L=2CL_1=2Aec(d,M)L_1$, we can estimate
	\begin{eqnarray}
	|\partial_x^k\partial_y^l (\chi_\alpha(x)g_\alpha(y))|&\leq & K((C/L)^{|\alpha|}M_{|\alpha|}^{-1}(L(2+A^{-1}))^{|k|}M_{|k|})\cdot(L_1^{|\alpha|}L_2^{|l|}M_{|\alpha|}\tilde{M}_{|l|})\\
	&\leq & K\cdot 2^{-|\alpha|}(2CL_1(2+A^{-1}))^{|k|}L_2^{|l|}M_{|k|}\tilde{M}_{|l|}
	\end{eqnarray}
	where $A,C,$ and $K$ are constants independent of $f,L_1,$ and $L_2$.
	
	Hence we have that $\|\chi_\alpha(x)g_\alpha(y)\|_{2CL_1(2+A^{-1}),L_2}\leq K\cdot 2^{-|\alpha|}$. It follows that 
	\begin{equation}
	f(x,y):=\sum_{\alpha\in \N^d}\chi_\alpha(x)g_\alpha(y)
	\end{equation}
	converges in the $\Cm$ sense, and satisfies $\partial_x^\alpha f(0,y)=g_\alpha(y)$ as required.
\end{proof}
Equipped with these tools, we are ready to prove Theorem \ref{whitneymainthm}.
\begin{proof}[Proof of Theorem \ref{whitneymainthm}]
	We begin by estimating the difference in Taylor expansions about different points in $K$.
	Using the identity
	\begin{equation}
	(T_x^{n}f)(z)-(T_y^{n}f)(z)=\sum_{|\alpha|\leq n} \frac{(z_1-x)^\alpha}{\alpha!}(R_y^nf)_{\alpha,0}(x,z_2)
	\end{equation}
	we can estimate
	\begin{eqnarray}
	& &\partial_z^{k,l}((T_x^{n}f)(z)-(T_y^{n}f)(z))\\
	&=& \sum_{|\alpha|\leq n-|k|}\frac{(z_1-x)^\alpha}{\alpha!}(R_y^{n}f)_{k+\alpha,l}(x)
	\end{eqnarray}
	using the assumed estimate \eqref{gevcond2} for $(R_y^{m,n}f)_{k,l}$. 
	This yields
	\begin{equation}
	\label{diffoftaylor}
	|\partial_z^{k,l}((T_x^{n}f)(z)-(T_y^{n}f)(z))|\leq AL_1^{n+1}M_{n+1}L_2^{|l|}\tilde{M}_{|l|}\frac{(|z_1-x|+|z_1-y|)^{n-|k|+1}}{(n-|k|+1)!}.
	\end{equation}
	We now invoke Proposition \ref{carlemanvariant}.
	For $x\in X$ consider the map $T_x:\Cm\rightarrow G_x$ given by $(T_xf)_\alpha(y):=f^{\alpha,0}(x,y)$ where the space $G_x$ consists of all multisequences of analytic functions $f_\alpha:Y\rightarrow \R$ satisfying $|f_\alpha|\leq AL_1^{|\alpha|}L_2^{|\beta|}M_{|\alpha|}\tilde{M}_{|\beta|}$ for some $A>0$.
	From the assumed estimate \eqref{gevcond} on $f^{\alpha,\beta}$, Proposition \ref{carlemanvariant} applies, and for each $x\in K$, we can find a function $f_x\in \Cm$ such that 
	
	
	
	
	\begin{equation}
	\label{fxsolvescarleman}
	\partial^{\alpha,\beta}_zf_x(x,z_2)=f^{\alpha,\beta}(x,z_2)
	\end{equation}
	for each $\alpha,\beta$. Moreover, the conclusion of Proposition \ref{carlemanvariant} implies that there exist constants $B=C_0A,K_1=C_1L_1,K_2=L_2>0$ such that the estimate
	\begin{equation}
	\label{boundedcarleman}
	|(\partial^{\alpha,\beta}_zf_x)(z)|\leq BK_1^{|\alpha|}K_2^{|\beta|}M_{|\alpha|}\tilde{M}_{|\beta|}
	\end{equation} 
	holds uniformly, where $C_j$ depend only on the dimensions of $X$ and $Y$ and the weight sequences $M_k,\tilde{M}_k$.
	
	
	Hence we can bound 
	\begin{equation}
	\partial_z^{k,l}(f_x(z)-(T_x^{m,n}f_x)(z))=(R^{m,n}f_x)_{k,l}(z)
	\end{equation}
	using the same calculation as in Proposition \ref{anistaylor}.
	We obtain
	\begin{eqnarray}
	|\partial_x^{k,l}(f_x(z)-(T_x^nf)(z))|&=&|(R^nf_x)_{k,l}(z)|\\
	&\leq& A(C_1L_1)^{n+1}M_{n+1}L_2^{|l|}\tilde{M}_{|l|}\frac{|z_1-x|^{n-|k|+1}}{(n-|k|+1)!}.
	\end{eqnarray}
	
	The upshot of this estimate is that we can replace $T_x^nf$ and $T_y^nf$ in \eqref{diffoftaylor} with $f_x$ and $f_y$ respectively. 
	That is, we have
	\begin{equation}
	\label{fxfyn}
	|\partial_z^{k,l}(f_x(z)-f_y(z))|\leq A(C_1L_1)^{n+1}M_{n+1}L_2^{|l|}\tilde{M}_{|l|}\frac{(|z_1-x|+|z_1-y|)^{n-|k|+1}}{(n-|k|+1)!}.
	\end{equation}
	We now fix $k,l$ and vary $n\geq k$ in order to optimise the upper bound \eqref{fxfyn}.
	By defining the quantity 
	\begin{equation}
	h(t):=\sup_{k\geq 0} \frac{k!}{t^k M_k}
	\end{equation}
	as in \cite{bruna} we obtain
	\begin{equation}
	\label{fxfy}
	|\partial_z^{k,l}(f_x(z)-f_y(z))|\leq A(C_1L_1)^{|k|}M_{|k|}L_2^{|l|}\tilde{M}_{|l|}h((C_1L_1)(|z_1-x_1|+|z_1-y|))^{-1}.
	\end{equation}
	by using property (3) following Definition \ref{anisspace}.
	
	The next step in the construction is to use Proposition \ref{pou} to piece together the functions $f_x$ using a $\mathcal{C}^\infty_{M}$ partition of unity subordinate to the cover arising from the decomposition of $X \setminus K$ by cubes in Proposition \ref{cubesprop}.
	Taking the collection $\{Q_j\}_{j\in\N}$ of cubes in $X=\mathbb{R}^d$ constructed by Proposition \ref{cubesprop}, we choose $x_j\in K$ such that $d(x_j,Q_j)=d(Q_j,K)$.
	Note that the conclusion of Proposition \ref{cubesprop} implies that 
	\begin{equation}
	|z-x_j|\sim d(z,K)
	\end{equation}
	for all $z\in Q_j^*$.
	Now taking $\phi_j$ as in Proposition \ref{pou}, we define:
	\begin{equation}
	\tilde{f}(z):= \begin{cases} f(z) &\mbox{if } z_1\in K \\ 
	\sum_i \phi_i(z_1)f_{x_j}(z) & \mbox{if } z_1\in X\setminus K.\end{cases}
	\end{equation}
	Note that since the partition of unity $\{\phi_j\}$ is locally finite, the function $\tilde{f}(z)$ is smooth in $(X\setminus K)\times Y$.
	It remains to check that $\tilde{f}$ is smooth elsewhere, and moreover that $\tilde{f}\in \Cm$.
	To this end, for $x\in K$ and $z_1\in X\setminus K$, we estimate
	\begin{equation}
	\partial_z^{\alpha,\beta}(\tilde{f}(z)-f_x(z))=\sum_{k\leq \alpha}\binom{\alpha}{k}\sum_i (\partial^k\phi_i)(z_1)\cdot \partial_z^{\alpha-k,\beta}(f_{x_i}(z)-f_x(z)).
	\end{equation}
	
	First we estimate the $k=0$ term. If $z_1\in \textrm{spt}(\phi_i)=Q_i^*$, we have 
	\begin{equation}
	d(z_1,x_i)\sim d(z_1,K)\leq d(z_1,x)
	\end{equation}
	and hence we have 
	\begin{equation}
	\label{kzero}
	|\sum_i \phi_i(z_1)\cdot \partial_z^{\alpha,\beta}(f_{x_i}(z)-f_x(z))|\leq A(C_1L_1)^{|\alpha|}M_{|\alpha|}L_2^{|\beta|}\tilde{M}_{|\beta|}h((C_1L_1)|z_1-x|)^{-1}
	\end{equation}
	from \eqref{fxfy}.
	
	We now estimate the terms with $|k|>0$. For $x\in X\setminus K$, we choose $\bar{x}\in K$ with $d(x,\bar{x})=d(x,K)$.
	Since $\sum_i \partial^k\phi_i=0$, we have
	\begin{equation}
	\sum_i (\partial^k\phi_i)(z_1)\cdot \partial_z^{\alpha-k,\beta}(f_{x_i}(z)-f_x(z))=\sum_i (\partial^k\phi_i)(z_1)\cdot \partial_z^{\alpha-k,\beta}(f_{x_i}(z)-f_{\bar{z_1}}(z)).
	\end{equation}
	Now as before, we exploit the fact that $d(z_1,x_i)\sim d(z_1,K)$ to bound
	\begin{equation}
	|\partial_z^{\alpha-k,\beta}(f_{x_i}(z)-f_{\bar{z_1}}(z))|\leq A(C_1L_1)^{|\alpha|-|k|}M_{|\alpha|-|k|}L_2^{|\beta|}\tilde{M}_{|\beta|}h((C_1L_1)d(z_1,K))^{-1}.
	\end{equation}
	Since $\log(M_j)$ is an increasing convex sequence with first term $0$, it is also superadditive, and we have $M_{|k|}M_{|l|}\leq M_{|k|+|l|}$. Hence for $|k|\geq 1$, we can use property (4) in Proposition \ref{pou} to conclude that
	\begin{equation}
	\left|\sum_i (\partial^k\phi_i)(z_1)\cdot \partial_z^{\alpha-k,\beta}(f_{x_i}(z)-f_x(z))\right|\leq AM_{|\alpha|}\tilde{M}_{|\beta|}(C_1L_1)^{|\alpha|-|k|}L_2^{|\beta|}\eta^{|k|}\frac{h(B\eta d(z_1,K))}{h((C_1L_1)d(z_1,K))}
	\end{equation}
	where $\eta$ remains to be chosen.
	Equation (15) from \cite{bruna} implies the existence of a constant $c>0$ such that
	\begin{equation}
	\frac{h(t)}{h(ct)}\leq \frac{A}{h(t)}
	\end{equation}
	for some $A>0$.
	Hence we choose $\eta =(C_1L_1)/cB$ to arrive at the estimate 
	\begin{equation}
	\label{knotzero}
	\left|\sum_i (\partial^k\phi_i)(z_1)\cdot \partial_z^{\alpha-k,\beta}(f_{x_i}(z)-f_x(z))\right| \leq A(C_1L_1)^{|\alpha|-|k|}L_2^{|\beta|}M_{|\alpha|}\tilde{M}_{|\beta|}\eta^{|k|}h((C_1L_1)|z_1-x|)^{-1}.
	\end{equation}
	Combining \eqref{kzero} and \eqref{knotzero}, we arrive at
	\begin{equation}
	\label{whitneykey}
	|\partial_z^{\alpha,\beta}(\tilde{f}(z)-f_x(z))|\leq AL_2^{|\beta|}M_{|\alpha|}\tilde{M}_{|\beta|}((C_1L_1)+\eta)^{|\alpha|}h((C_1L_1)|z_1-x|)^{-1}
	\end{equation}
	for $z\in (X\setminus K) \times Y$.
	
	The estimate \eqref{whitneykey} is key to proving $\tilde{f}\in\mathcal{C}^\infty(X\times Y)$ (and that the derivatives coincide with the those given by the jet $f^{\alpha,\beta}$), as well as the subsequent deduction of $\mathcal{C}^\infty_{M,\tilde{M}}$ regularity.
	We write
	\begin{equation}
	\tilde{f}^{\alpha,\beta}(z):=\begin{cases} \partial_z^{\alpha,\beta}\tilde{f}(z) &\mbox{if } z_1\in X\setminus K \\ 
	f^{\alpha,\beta}(z) & \mbox{if } z_1\in K.\end{cases}
	\end{equation}
	The smoothness of each $\tilde{f}^{\alpha,\beta}:X\times Y\rightarrow \R$ readily follows from the fact that  each $f^{\alpha,\beta}:K\times Y\rightarrow \R$ is smooth in $y$, together with the estimate
	\begin{equation}
	\label{smoothestimate}
	|\tilde{f}^{\alpha,\beta}(z)-\partial_z^{\alpha,\beta}T_x^mf(z)|=o(|z_1-x|^{m-|\alpha|}).
	\end{equation}
	For $z$ with $z_1\in K$, the estimate \eqref{smoothestimate} comes immediately from \eqref{gevcond2} on $K\times Y$. Otherwise, it is a consequence of the estimate \eqref{whitneykey}, the defining property  \eqref{fxsolvescarleman} of the functions $f_x$, and the fact that the function $h(t)$ increases faster than any polynomial in $t^{-1}$ as $t\rightarrow 0$.
	
	Finally, we need to check $\mathcal{C}^\infty_{M,\tilde{M}}$ regularity. That is, we need to verify that the Gevrey estimate
	\begin{equation}
	\label{endofwhitney}
	\|f\|_{C_1L_1,L_2}\leq C_0A.
	\end{equation}
	for some constants $C_0,C_1$ dependent only on the dimensions of the spaces $X$ and $Y$ and the weight sequences $M_k,\tilde{M}_k$.
	In light of \eqref{gevcond}, it only remains to prove \eqref{endofwhitney} on $(X\setminus K)\times Y$, and by multiplication by a cutoff function we may assume $d(z_1,K)$ is bounded.
	Then, by applying \eqref{whitneykey} with $x=\bar{z}_1$ we can further reduce the problem to verifying \eqref{endofwhitney} for $f_x$, uniformly in $x\in K$. However this was established earlier in \eqref{boundedcarleman}. 
	Hence, the proof is complete.
\end{proof}

\bibliographystyle{plain}
\bibliography{non.qe}
\end{document}